\documentclass[reqno]{amsart}

\usepackage{graphicx, amsmath, amssymb, amsthm, tikz}

\newtheorem{theorem}{Theorem}

\newtheorem{lemma}{Lemma}
\newtheorem{proposition}{Proposition}

\DeclareMathOperator*{\Vdc}{Vdc}

\begin{document}

\title[]{Approximately jumping towards the Origin}

\author[]{Alex Albors}
\author[]{Fran\c{c}ois Cl\'ement}
\author[]{Shosuke Kiami}
\author[]{Braeden Sodt}
\author[]{Ding Yifan}
\author[]{Tony Zeng}

\address{Department of Mathematics, University of Washington, Seattle, WA 98195, USA}
\email{aalbors@uw.edu }
\email{fclement@uw.eu }
\email{shokiami@uw.edu}
\email{bsodt@uw.edu }
\email{yifan29@uw.edu}
\email{txz@uw.edu}

\thanks{This work was carried out under the umbrella of the WXML (Washington Experimental Mathematics Laboratory) whose support is gratefully acknowledged. The authors acknowledge useful discussions with Douglas Rizzolo and Stefan Steinerberger.}

\begin{abstract} 
Given an initial point $x_0 \in \mathbb{R}^d$ and a sequence of vectors $v_1, v_2, \dots$ in $\mathbb{R}^d$, we define a greedy sequence by setting $x_{n} = x_{n-1} \pm v_n$ where the sign is chosen so as to minimize $\|x_n\|$. 
We prove that if the vectors $v_i$ are chosen uniformly at random from $\mathbb{S}^{d-1}$ then elements of the sequence are, on average, approximately distance $\|x_n\| \sim \sqrt{\pi d/8}$ from the origin. We show that the sequence $(\|x_n\|)_{n=1}^{\infty}$ has an invariant measure $\pi_d$ depending only on $d$ and we determine its mean and study its decay for all $d$. We also investigate a completely deterministic example in $d=2$ where the $v_n$ are derived from the van der Corput sequence. Several additional examples are considered.
\end{abstract}

\maketitle

\vspace{-10pt}

\section{Introduction and Results}
\subsection{Introduction}
Our results are motivated by a beautiful and fairly recent result of Bettin-Molteni-Sanna \cite{Bettin}. Pick some $x > 0$ and insert signs into the standard harmonic series $\sum 1/n$ in a greedy way such that at each step we are as close as possible to $x$. More formally, we set $x_1 = 1$ and, for all $n \geq 2$,
$$ x_{n} = \begin{cases} x_{n-1} + \frac{1}{n} \qquad &\mbox{if}~x_n \leq x \\ x_{n-1} - \frac{1}{n} \qquad &\mbox{if}~a_n > x. \end{cases}$$
The divergence of the harmonic series implies that, after a finite number of steps, the sequence is always going to be uniformly close to $x$ and $|x_n - x| \leq 2/n$. Bettin-Molteni-Sanna \cite{Bettin} prove that for a generic $x \in \mathbb{R}_{>0}$, there are subsequences along which $x_n$ is \textit{much} closer to $x$: for any $c> 0$ there exists a subsequence of positive density depending on $c$ for which $|x_{n_k} - x| \leq n_k^{-c}$. 

\begin{center}
\begin{figure}[h!]
\begin{tikzpicture}
\node at (0,0) {\includegraphics[width=0.3\textwidth]{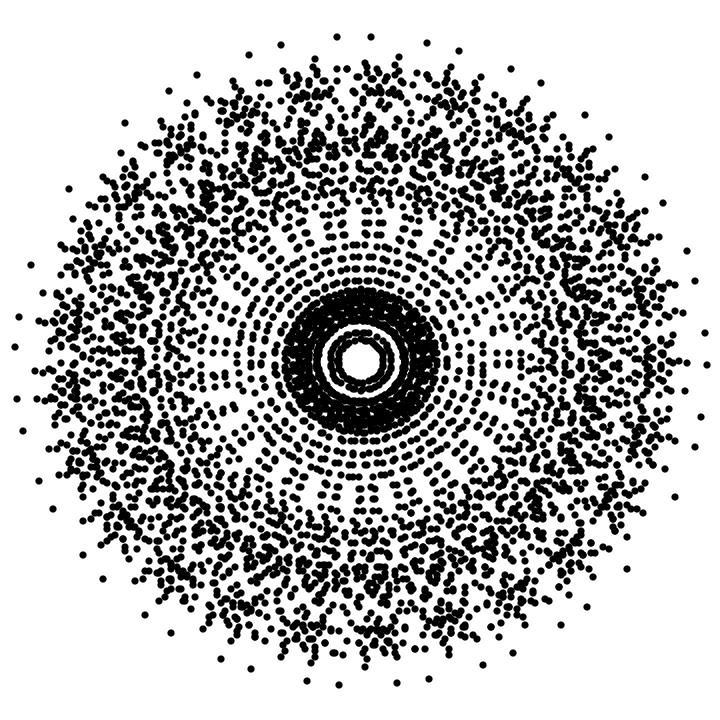}};
\node at (6,0) {\includegraphics[width=0.3\textwidth]{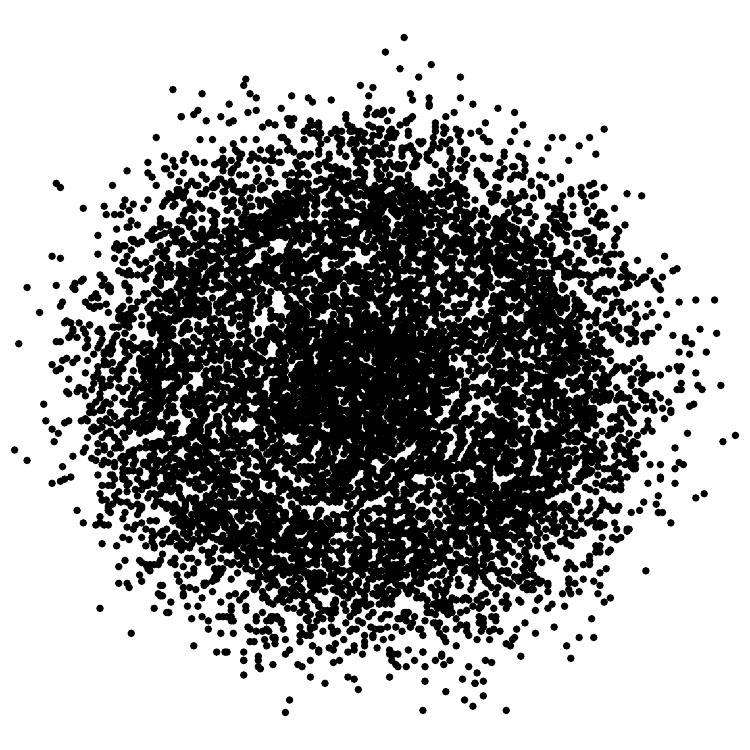}};
\end{tikzpicture}
\caption{Left: starting in $0.0001 + 5i$ with $\alpha = 1.0415 \sqrt{2}$ and $v_n = e^{ i  \alpha n}$. Right: using the vector $v_n = e^{in^2}$.}
\end{figure}
\end{center}

Steinerberger-Zeng \cite{steinzeng} considered a two-dimensional variant (see Fig. 1): identifying $\mathbb{R}^2 \cong \mathbb{C}$, they considered $z_{n+1} = z_n \pm e^{i \alpha n}$,
whichever is closer to the origin (where $\alpha$ is an irrational real).
More formally, the definition is
$$ z_{n} = \begin{cases} z_{n-1} +  e^{i \alpha n} \qquad &\mbox{if}~|z_{n-1} + e^{i \alpha n}|  < |z_{n-1} - e^{i \alpha n}|  \\ z_{n-1} -  e^{i \alpha n} \qquad &\mbox{if}~|z_{n-1} + e^{i \alpha n}|  > |z_{n-1} - e^{i \alpha n}|  \end{cases}$$
with $z_n$ being undefined if both norms are the same (which happens very rarely). This dynamical system can also be interpreted as the action of a particular piecewise isometry \cite{goetz2,goetz3}. It gives rise to the most curious patterns: in particular, the choice of signs appears to be ultimately periodic in which case the sequence lies on a finite union of circles with the same center (see Fig. 1). The special nature of the sequence $e^{i \alpha n}$ is heavily exploited. This leads to an interesting general problem. Given $x \in \mathbb{R}^d$ and a sequence of vectors $v_1, v_2, \dots \in \mathbb{R}^d$, suppose we consider the
$$ x_{n} = \begin{cases} x_{n-1} + v_n \qquad &\mbox{if}\quad \|x_{n-1}+ v_n\| < \|x_{n-1} - v_n\| \\ x_{n-1} - v_n \qquad &\mbox{if}\quad \|x_{n-1}- v_n\| < \|x_{n-1} + v_n\|. \end{cases}$$
What can we expect? It is clear, a priori, that one might expect a general trend towards the origin. It is not clear whether one would expect much more than that.

\begin{center}
\begin{figure}[h!]
\begin{tikzpicture}
\node at (0,0) {\includegraphics[width=0.25\textwidth]{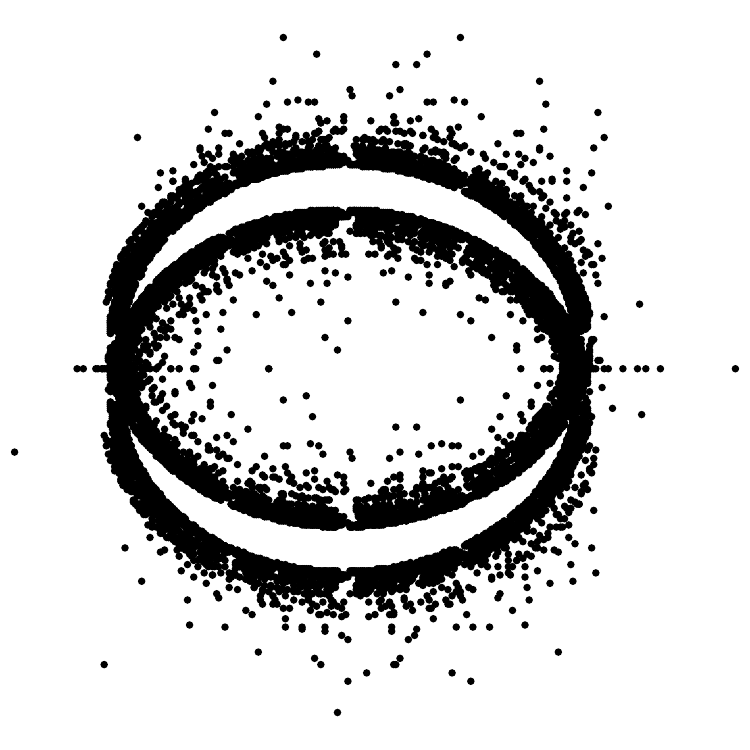}};
\node at (4.5,0) {\includegraphics[width=0.25\textwidth]{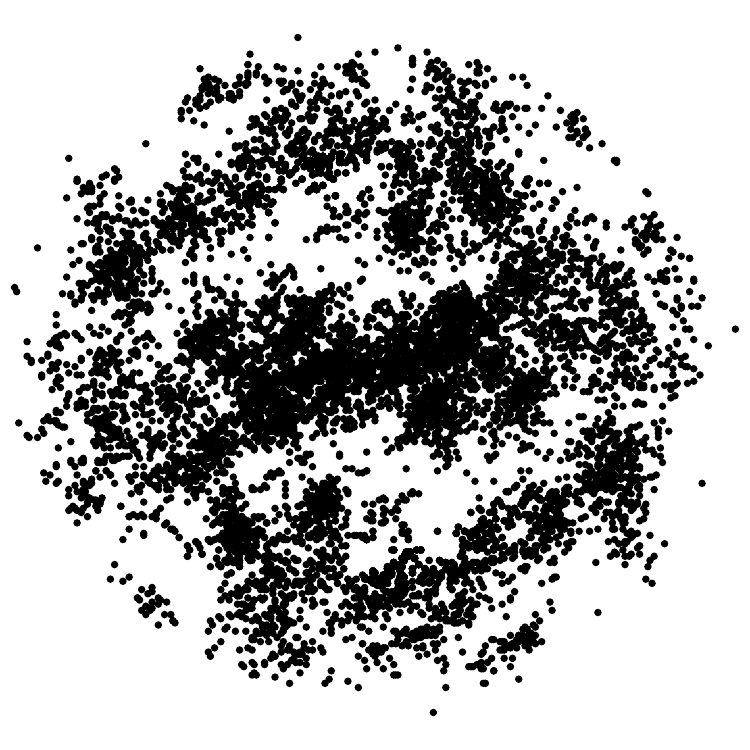}};
\node at (9,0) {\includegraphics[width=0.25\textwidth]{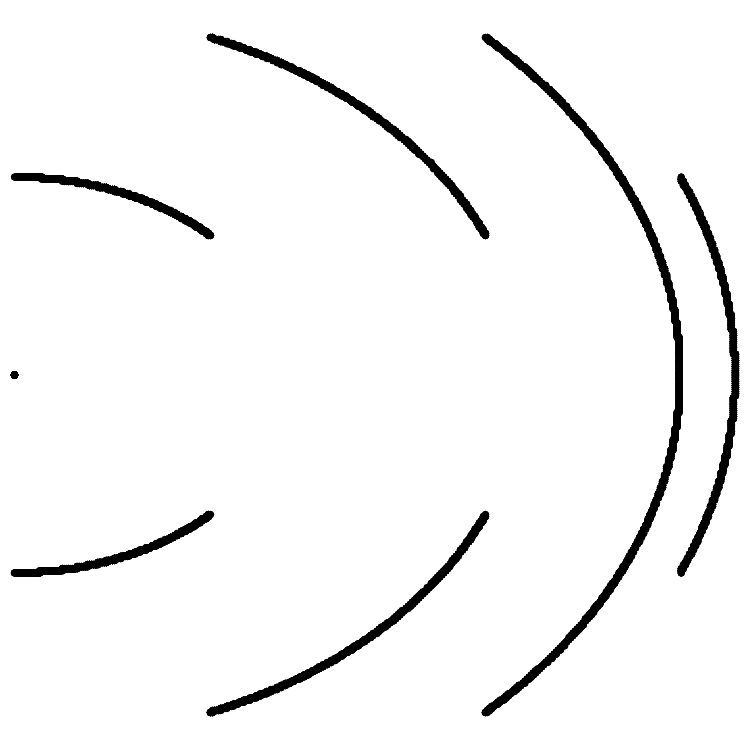}};
\end{tikzpicture}
\label{fig:examplesvdc}
\caption{Using the Farey sequences (left), the van der Corput sequence in base 3 (middle) and in base 8 (right).}
\end{figure}
\end{center}

This is already highly nontrivial in the two-dimensional case. Restricting ourselves to vectors (complex numbers) of unit length one could consider sequences $(a_n)_{n=1}^{\infty}$ and then set $v_n = e^{2 \pi i a_n}$. Fig. 2 shows three such examples. The first one comes from concatenating the Farey fractions which results in
$$ 0,1, 0, \frac{1}{2}, 1, 0, \frac13, \frac12, \frac23, 1, 0, \frac14, \frac13, \frac12, \frac23, \frac34, 1, 0, \frac15, \frac14, \frac13, \frac25, \frac12, \dots$$
The other two examples are derived from the van der Corput sequence which is defined as follows: to compute $x_n \in [0,1]$, write $n$ in base $b$, reflect its digits around the comma and turn it back into a real number. The van der Corput sequence in base 2 begins
$$ \frac12, \frac14, \frac34, \frac18, \frac58, \frac38, \frac78, \dots$$
Three arising dynamical systems are shown in Figure 2. We will be able to say some things about the van der Corput system, see \S 1.3, we do not have any rigorous results of any type for the Farey system.

\subsection{The random case}
It is a natural assumption that for `many' choices of vectors, the arising dynamical system will not have any type of structure: it will be `random'. This leads to a natural question: what do we observe when the sequence is truly random?
To motivate the question, we consider the two sequences 
$$ z_n =  z_{n-1} \pm e^{2 \pi i \sqrt{2} n^2} \qquad \mbox{and} \qquad z_n =  z_{n-1} \pm e^{2 \pi i \sqrt{2} n^3}.$$
When drawing a great many points, it is difficult to see a difference, they each approximately fill out a disk. When keeping track of the distance to the origin, however, a difference seems to reveal itself (see Fig. 3).

\begin{center}
\begin{figure}[h!]
\begin{tikzpicture}
\node at (0,0) {\includegraphics[width=0.45\textwidth]{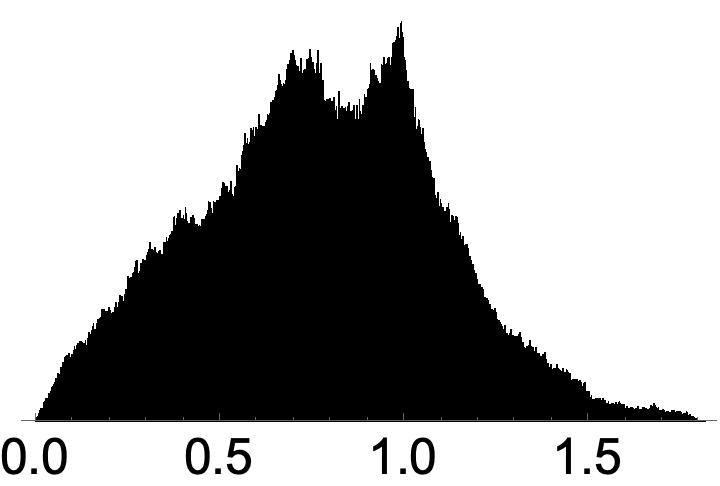}};
\node at (6,0) {\includegraphics[width=0.45\textwidth]{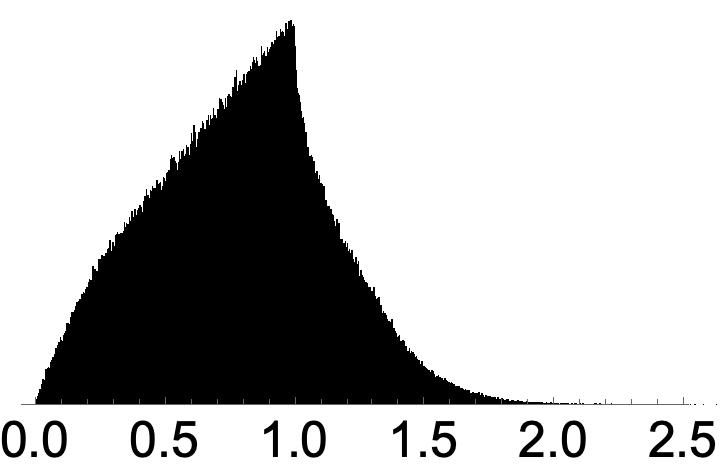}};
\end{tikzpicture}
\caption{Distribution of distances from the origins for the first million terms of $\pm e^{2 \pi i \sqrt{2} n^2}$ (left) and $\pm e^{2 \pi i \sqrt{2} n^3}$ (right).}
\end{figure}
\end{center}

As it turns out, the distribution for $\pm e^{2 \pi i \sqrt{2} n^3}$ coincides with what one gets if one replaces $\pm e^{2 \pi i \sqrt{2} n^3}$ by $\pm e^{i X_n}$ where $X_n$ is chosen uniformly at random with respect to $[0, 2\pi]$. The distribution of $ \sqrt{2} n^3 \mod~2\pi$ is sufficiently `random' to emulate the random case while, in contrast, the sequence $ \sqrt{2} n^2 \mod~2\pi$ is not. Formally, we define the random sequence by $x_0 \in \mathbb{R}^d$
and
$$ x_{n} = \begin{cases} x_{n-1} + v_n \qquad &\mbox{if}\quad \|x_{n-1}+ v_n\| < \|x_{n-1} - v_n\| \\ x_{n-1} - v_n \qquad &\mbox{if}\quad \|x_{n-1}- v_n\| < \|x_{n-1} + v_n\|, \end{cases}$$
where $v_n$ is chosen uniformly at random from $\mathbb{S}^{d-1}$. We are interested in the asymptotic density of $x_n$. By decoupling into radial and angular part, it is clear that if we prove sufficiently strong upper bounds on $\|x_n\|$, then the distribution is going to be radial (the angular part can be considered to be a random walk on $\mathbb{S}^{d-1}$), it thus suffices to understand $\|x_n\|$.

\begin{theorem} For each dimension $d \geq 2$, the random variable $\|x_n\|$ has a unique, invariant measure $\pi_d$ on $[0, \infty]$. This invariant measure $\pi_d$ has expectation 
$$ \int_{0}^{\infty} x ~d\pi_d(x) = \frac{\sqrt{\pi}}{2} \frac{ \Gamma(\frac{d+1}{2})}{\Gamma(\tfrac{d}{2})} = \sqrt{\frac{\pi}{8}} \sqrt{d} + o(\sqrt{d})$$
which is asymptotically $(\sqrt{\pi/8} + o(1)) \cdot \sqrt{d}$.
The invariant measure $\pi_d$ has faster than Gaussian decay: for every $\alpha > 0$
$$ \int_{0}^{\infty} e^{\alpha x^2} d\pi_d(x) < \infty.$$
\end{theorem}

\textbf{Remarks.} Several remarks are in order.
\begin{enumerate}
\item This result can be seen as yet another illustration of the concentration of measure phenomenon \cite{ledoux}. When we are in $x_n \neq 0$, then most random vectors $v$ are actually going to be `nearly' orthogonal to $-x_n$ and thus $x_n \pm v$ is going to be typically further away from the origin than $x_n$ as long as $\|x_n\|$ is small. Once $\|x_n\|$ is large, even a small inner product can usually be used to decrease the norm. This process self-stabilizes around $\|x_n\| \sim c_d\sqrt{d}$ and this is where $\|x_n\|$ ends up spending most of its time.
\begin{center}
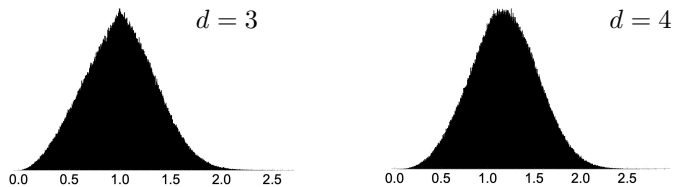
\begin{figure}[h!]
\begin{tikzpicture}
\node at (-3,0) {};
\node at (0,0) {\includegraphics[width=0.3\textwidth]{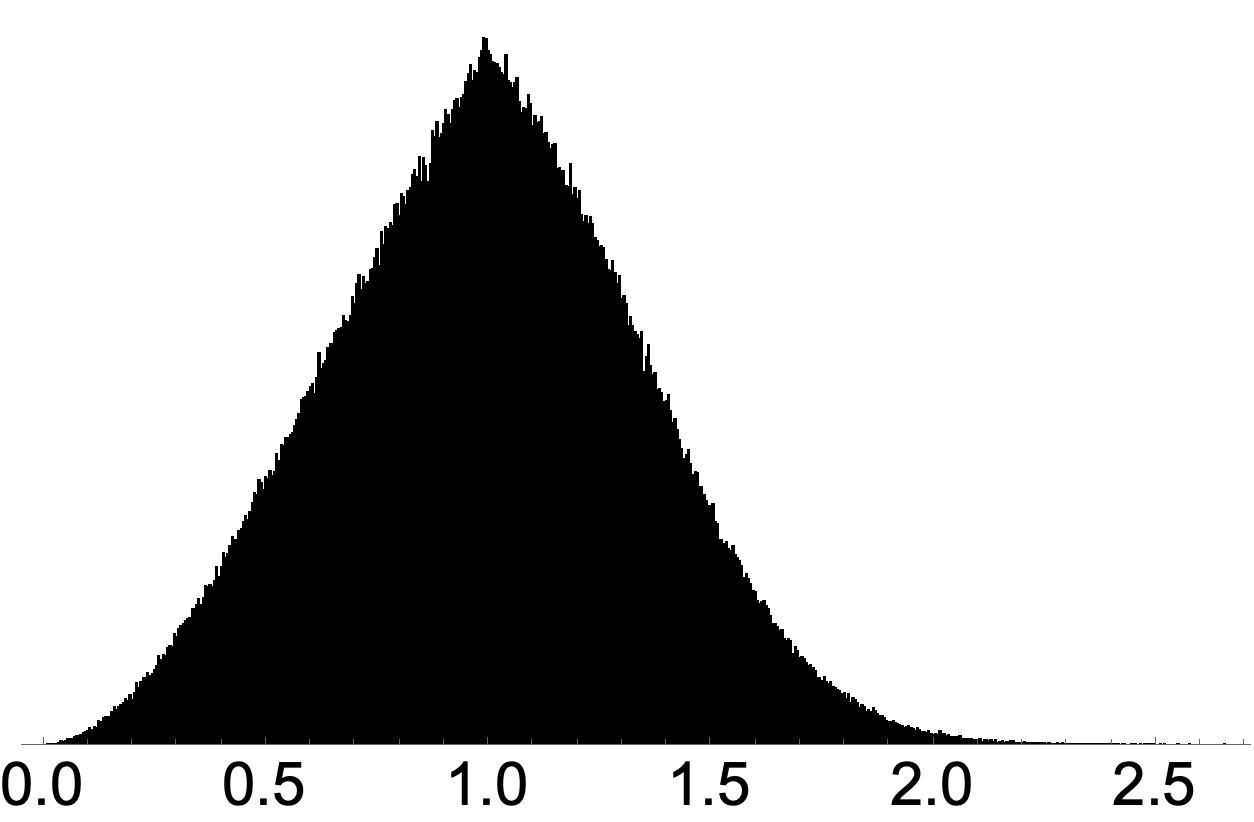}};
\node at (5,0) {\includegraphics[width=0.3\textwidth]{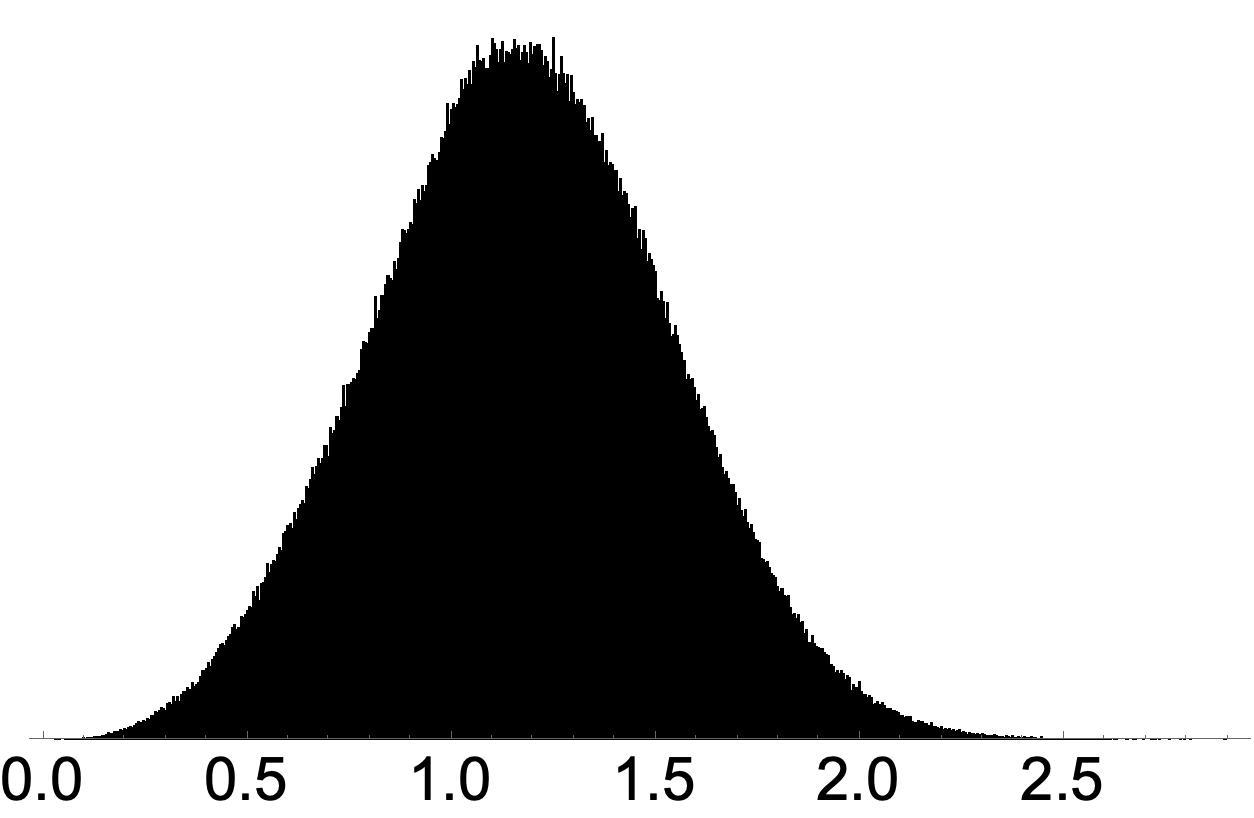}};
\node at (1,1) {$d=3$};
\node at (6.5, 1) {$d=4$};
\end{tikzpicture}
\caption{Distribution of distances to the origin for $d =3, 4$.}
\end{figure}
\end{center}
\item We may think of the process as a random walk on $[0,\infty]$. The main difficulties are that it's not infinitesimally generated, it jumps, and that the jumps are of bounded size. We will prove that it is an aperiodic, recurrent Harris chain; the decay estimate is achieved by the construction of a suitable Lyapunov function.
\item The Lyapunov function implies that $\pi_d$ decays quite quickly once one exceeds the expectation; one could build a similar Lyapunov function to show that it also decays quite quickly for values smaller than the expectation. We have not pursued this line of research because something stronger is true: none of the distributions $\pi_d$ are Gaussian, however, it appears that, even in relatively low dimensions, one has
$$ \pi_d  \quad \sim \quad \frac{e^{-(x- \mathbb{E}\pi_d)^2}}{\sqrt{\pi}}.$$
In particular, it appears as if the limiting distribution has a variance that is (asymptotically) constant.
It would be interesting to make this precise.
\end{enumerate}

\subsection{Deterministic Constructions.} We will now move towards the deterministic setting. Steinerberger-Zeng \cite{steinzeng} showed a remarkably degree of regularity arising from the choice of vectors $v_n = \exp(2 \pi i \alpha n)$ with $\alpha$ irrational. The sequence $n\alpha~\mbox{mod}~1$ is also known as the Kronecker sequence and is one of the best known examples of a sequence with a very regular distribution \cite{Nie92}. This motivated us to consider the canonical example of such a sequence $(a_n)_{n\in \mathbb{N}}$: the van der Corput sequence. This classical sequence has been extensively studied since its introduction by Johannes van der Corput \cite{vdc} in the 1930s. It is traditionally defined in base 2, but can be defined in any base $b$ in the following way. Given an integer $n$, write 
$$ n=\sum_{k=0}^{\ell}a_k b^k \qquad \mbox{where the `digits' satisy} \quad a_k \in \{0,\ldots,b-1\}.$$ The $n$-th point $\Vdc_b(n)$ of the base $b$ van der Corput sequence is defined as
$$\Vdc_b(n):=\sum_{k=0}^{\ell}a_k b^{-k-1}.$$
This is easiest to understand with an example. We already saw the first few terms of van der Corput sequence in base 2 above
$$ \frac12, \frac14, \frac34, \frac18, \frac58, \frac38, \frac78, \dots$$
More generally, we see that $a_0 =0$, $a_1 = 1/b$ and $a_b = 1/b^2$. As above, we now consider the sequence
$$ z_{n} = \begin{cases} z_{n-1} +  e^{2 \pi i \Vdc_{b}(n)} \qquad &\mbox{if}~|z_{n-1} + e^{2 \pi i \Vdc_{b}(n)}|  < |z_{n-1} -  e^{2 \pi i \Vdc_{b}(n)}|  \\ z_{n-1} -   e^{2 \pi i \Vdc_{b}(n)} \qquad &\mbox{if}~|z_{n-1} +  e^{2 \pi i \Vdc_{b}(n)}|  > |z_{n-1} -  e^{2 \pi i \Vdc_{b}(n)}|  \end{cases}$$

We start by showing that if we start anywhere in the plane and use the van der Corput sequence in base 2, then one reaches a fixed neighborhood of the origin in a number of steps proportional to the original distance from the origin.

\begin{theorem}\label{th:general}
    Let $b=2$ and $z_{-1} \in \mathbb{R}$. If the sequence is not indeterminate by then, it takes at most 
    $O(|z|)$ steps of the sequence to reach $B(0,\sqrt{2})$.
\end{theorem}
While similar results with varying ball sizes could be obtained for different $b$, we limit our proofs to base 2 and some insights into the base 3 case. We also show that for odd $b$ there exist $z_{-1} \in \mathbb{R}^2$ such that $z_{kb-1}=z_{-1}$ for any $k \in \mathbb{N}$ and the sign choice at all steps is always the same. We call such positions \emph{periodic starts}.

\begin{theorem}\label{th:vdc2}
Let $b=2$.  Any initial value in $B(0,1)\backslash \{(0,0)\}$ is a periodic start. When $z_{n}$ enters the unit ball and $n$ is odd, then subsequent elements of the sequence are contained in a point and a semicircle centered on that point of radius 1.
\end{theorem}

\begin{figure}[h!]
    \centering
    \includegraphics[width=0.5\linewidth]{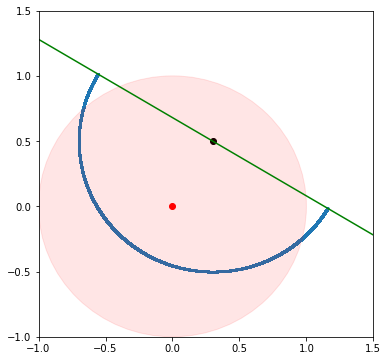}
    \caption{The sequence associated to base 2 for $z_{-1} \in B(0,1)\backslash \{(0,0)\}$. In red, the origin and $B(0,1)$. $z_{-1}$ is in black, and all other distinct points in the sequence are in blue. All points are in the semi-circle defined by the green line .}
    \label{fig:base2}
\end{figure}

As the base $b$ increases, things become more complicated. When $b$ is even, we can prove that no periodic starts exist. In that sense, the case $b=2$ is special. If $b \geq 5$ is odd, then periodic starts always exist and we can identify two regions that yield periodic starts.

\begin{figure}[h!]
    \centering
    \includegraphics[width=0.3\linewidth]{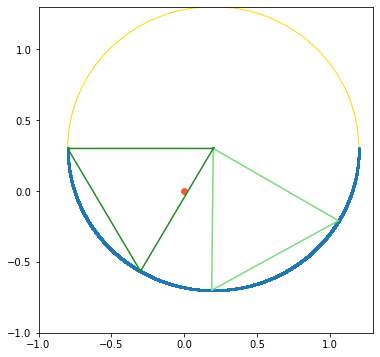}
    \includegraphics[width=0.3\linewidth]{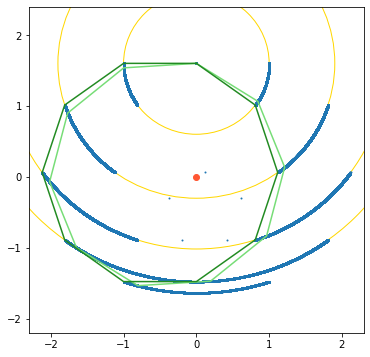}
    \includegraphics[width=0.3\linewidth]{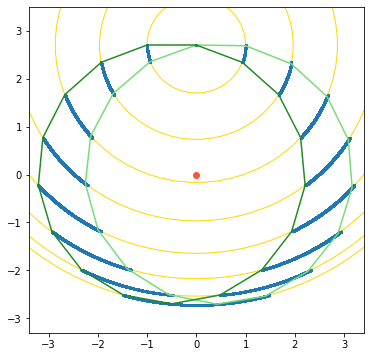}
    \caption{The polygons corresponding to the first and $100^{th}$ set of $b$ elements of the sequence associated with base $b$ van der Corput, for $b=3$, 10 and 17 (left to right). $200\,000$ points are computed.}
    \label{fig:poly_intro}
\end{figure}

\begin{theorem}\label{th:polygon}
    Let $ b\geq 4$. Then, considering $(a_n)_{n \in \mathbb{N}}=\Vdc_b(n)$,
    \begin{enumerate}
        \item if $b$ is even, no periodic starts exist.
        \item if $b$ is odd, the periodic starts form two open triangles.
    \end{enumerate}
\end{theorem}
Empirically, periodic starts also exist when $b=3$ but that argument would require additional work. We conclude by showing that periodic starts lead to relatively uninteresting sequences: they trace out a regular polygon in a periodic manner, partially rotated around $z_{-1}$.

\begin{theorem}\label{th:polygonStruct}
Let $b \geq 2$. For any periodic start $z_{-1} \in \mathbb{R}^2$, the arising sequence forms rotated copies of a regular $b-$sided polygon.
\end{theorem}
In the even $b$ case, this characterization holds for a certain number of steps in the sequence that depends on $z_{-1}$, given in Proposition~\ref{prop:nomove}.

\begin{center}
\begin{figure}[h!]
\begin{tikzpicture}
\node at (0,0) {\includegraphics[width=0.35\textwidth]{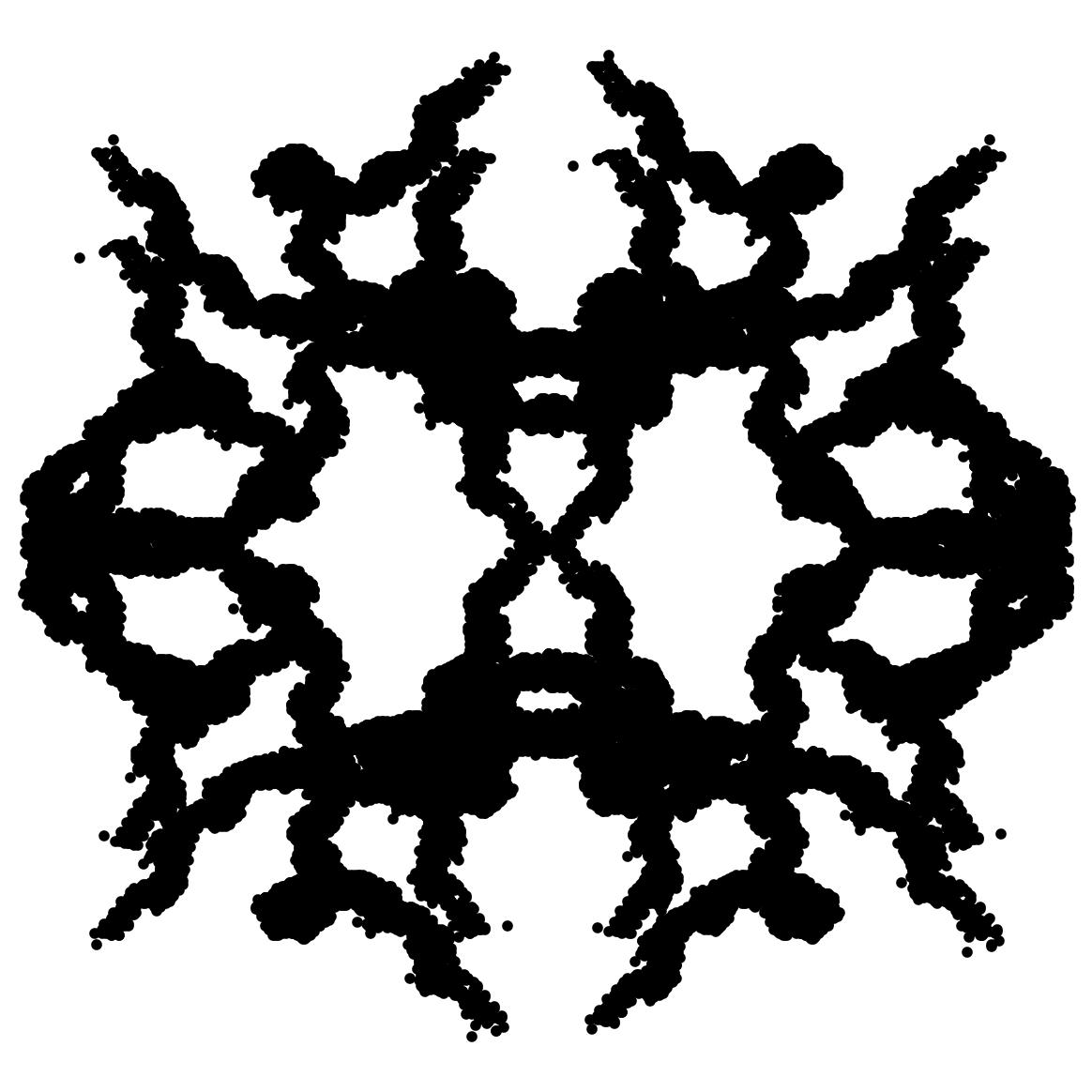}};
\node at (6,0) {\includegraphics[width=0.4\textwidth]{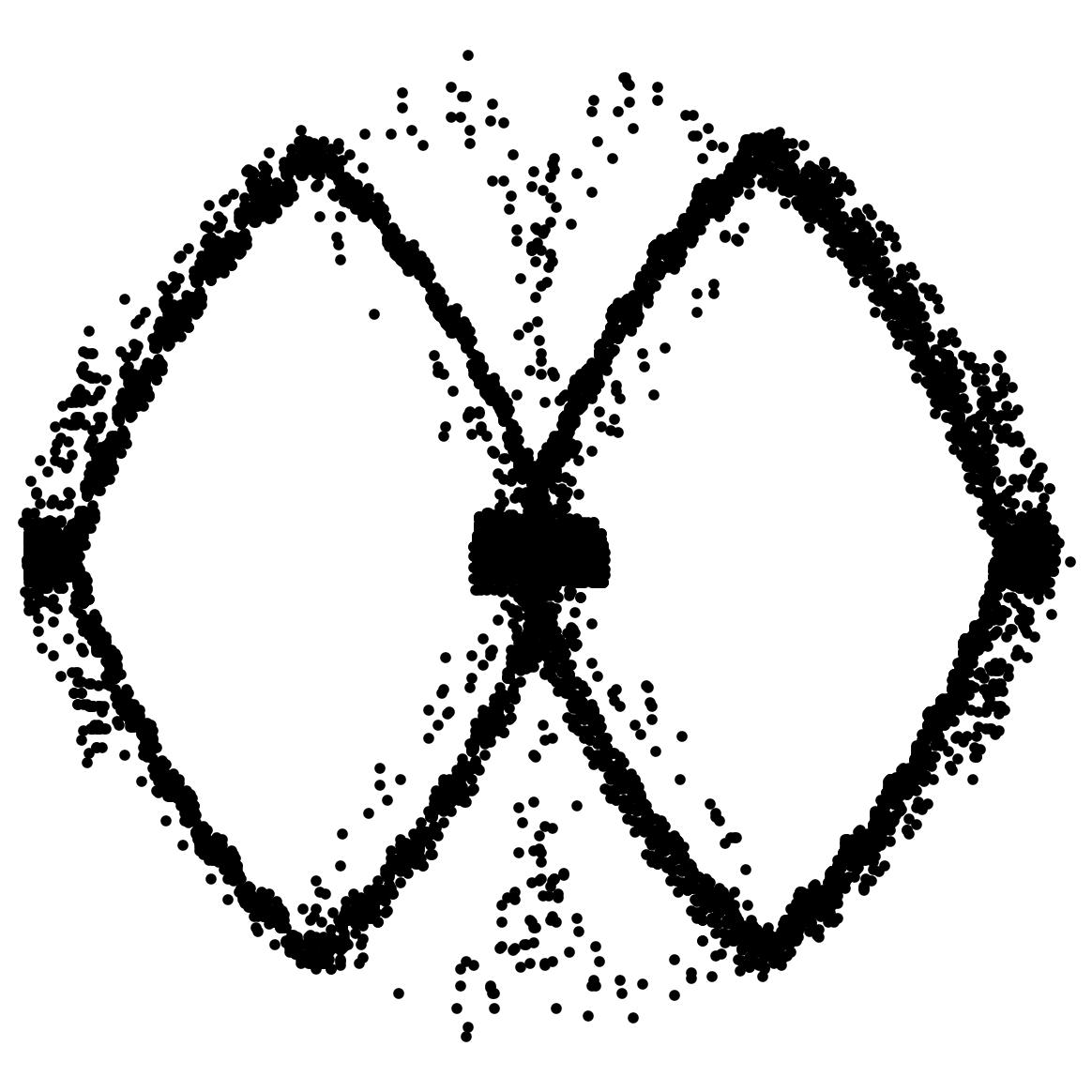}};
\end{tikzpicture}
\label{fig:examples}
\caption{Left: $v_n = \exp(2 \pi i \| \sqrt{3} n\|)$. Right: $v_n = \exp(2 \pi i \| \sqrt{10} n\|)$.}
\end{figure}
\end{center}

\subsection{More open problems.} Needless to say, the results presented up to now contain their fair number of associated problems. However, there are more. Indeed, it is worth noting that these types of greedy sequences
$$ x_{n} = \begin{cases} x_{n-1} + v_n \qquad &\mbox{if}\quad \|x_{n-1}+ v_n\| < \|x_{n-1} - v_n\| \\ x_{n-1} - v_n \qquad &\mbox{if}\quad \|x_{n-1}- v_n\| < \|x_{n-1} + v_n\| \end{cases}$$
appear to give rise to the most peculiar patterns. An example, starting in $x_0 = (0.01, 0.01)$ and using the sequence of vectors $v_n = \exp(2 \pi i \| \sqrt{3} n\|)$ with $\|\cdot\|$ denoting the distance to the nearest integer gives rise to points in a strange shape shown in Figure 7 (left).  Taking $v_n = \exp(2 \pi i \| \sqrt{10} n\|)$
gives rise to the shape shown in Figure 7 (right), both remain unexplained.

\begin{center}
\begin{figure}[h!]
\begin{tikzpicture}
\node at (0,0) {\includegraphics[width=0.35\textwidth]{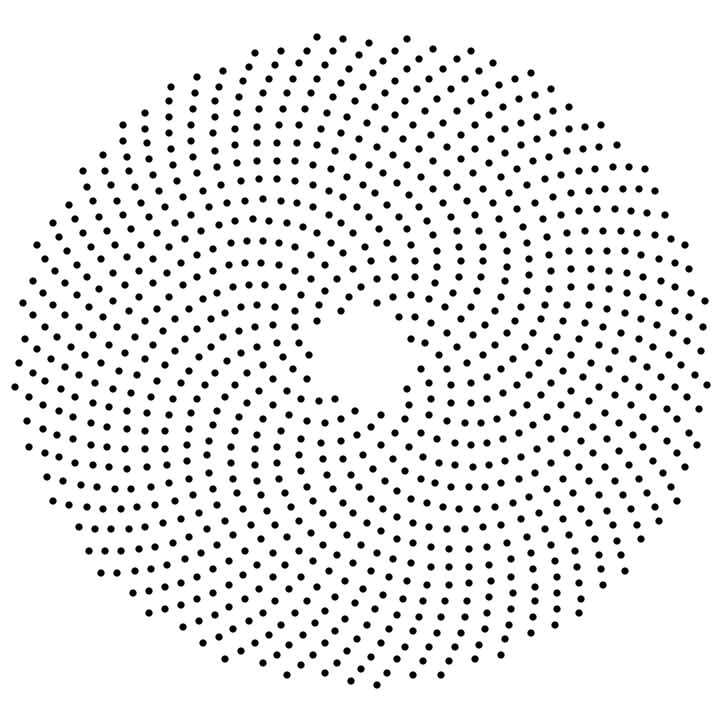}};
\node at (6,0) {\includegraphics[width=0.4\textwidth]{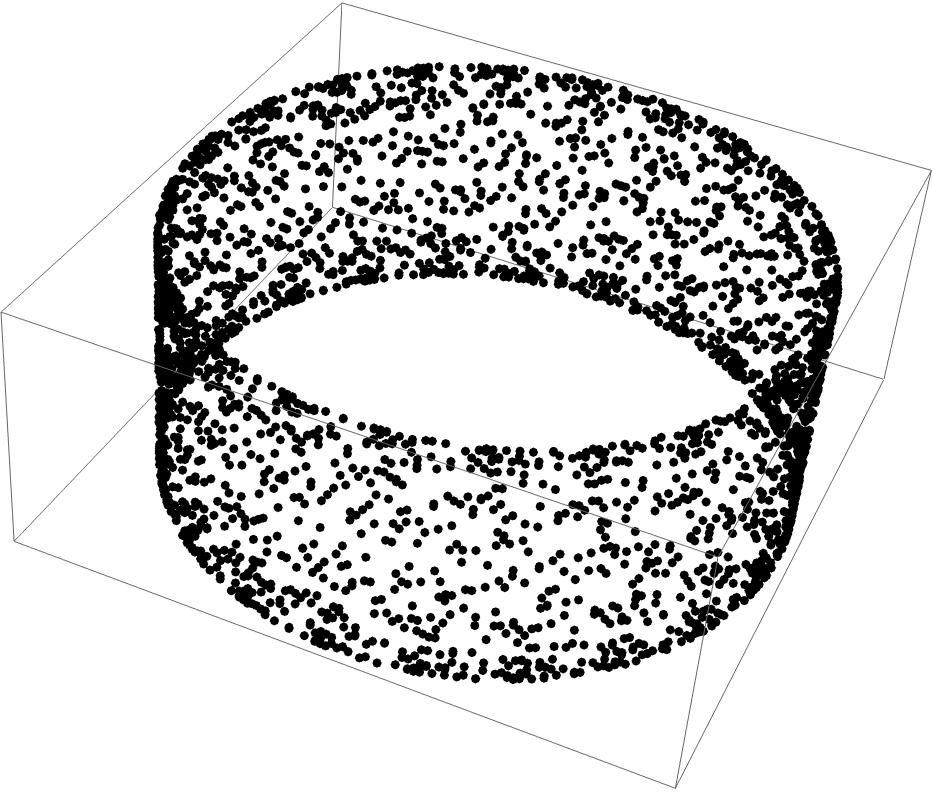}};
\end{tikzpicture}
\label{fig:examples}
\caption{Left: $v_n = \sqrt{n}\exp(2 \pi i \sqrt{2} n)$.  Right: $v_n = (\cos{n}, \sin{n}, \cos{\sqrt{n}})$.}
\end{figure}
\end{center}

All of these examples are two-dimensional with a sequence of vectors $(v_n)_{n=1}^{\infty}$ that has unit norm. We conclude by giving two examples that do not follow these constraints. First, we consider the sequence of vectors $v_n = \sqrt{n}\exp(2 \pi i \sqrt{2} n)$. The second example takes place in three dimensions with $v_n = (\cos{n}, \sin{n}, \cos{\sqrt{n}})$. We believe that these examples are indicative of a surprising richness and versatility of dynamical systems of this type, many open problems remain.

\section{Proof of Theorem 1}

\subsection{A stochastic representation} 
The purpose of this section is to give a slightly different way of phrasing the process. We can assume, using invariance under rotations, that
$$ x_n = (\|x_n\|, 0, 0, 0, \dots)$$
and that $x_{n+1} = x_n \pm v$ where $v = (v_1, \dots, v_d)$ is chosen uniformly at random from $\mathbb{S}^{d-1}$. Then
$$  x_n \pm v = (\|x_n\| \pm v_1,v_2,v_3,\dots,v_d)$$
and
$$ \|x_n \pm v\|^2 = (\|x_n\| \pm v_1)^2 + v_2^2 + \dots + v_d^2.$$
In this representation, it is clear that we will choose the sign so that 
\begin{align*}
 \|x_{n+1}\|^2 &= (\|x_n\| - |v_1|)^2  + v_2^2 + \dots + v_d^2 \\
 &= \|x_n\|^2 - 2 \cdot |v_1| \cdot \|x_n\| + 1.
 \end{align*}
 A classic way to sample uniformly at random from $\mathbb{S}^{d-1}$ is to pick $d$ independent $\mathcal{N}(0,1)$ Gaussians $(\gamma_1, \dots, \gamma_d)$ and then to normalize that vector. The first coordinate is distributed like
 $$ \frac{\gamma_1}{(\gamma_1^2 + \gamma_2^2 + \dots + \gamma_d^2)^{1/2}} =  \frac{\gamma_1}{(\gamma_1^2 + \chi^2(d-1))^{1/2}},$$
 where $\chi^2(d-1)$ is a chi-squared random variable of degree $d-1$. We may thus equivalently rewrite the process as
 $$  \|x_{n+1}\| = \left( \|x_n\|^2 - 2 \cdot \frac{|\gamma_1|}{(\gamma_1^2 + \chi^2(d-1))^{1/2}} \cdot \|x_n\| + 1 \right)^{1/2}$$
 which is a useful way of rewriting things.

 \subsection{Existence of a limiting measure $\pi$.}
 \begin{proof} Here we make use of the Harris chain framework \cite{harris}. We fix the dimension $d$.  A Markov chain on $\Omega = [0,\infty)$ with a stochastic transition kernel is called a Harris chain if there exists a set $A \subset \Omega$, a constant $\varepsilon > 0$ and a probability measure $\rho(\Omega) = 1$ such that
 \begin{enumerate}
 \item the hitting time $ \tau_A = \inf\left\{n \geq 0: X_n \in A \right\}$ is finite almost surely
 $$ \mathbb{P}(\tau_A < \infty| x_0 = x) = 1 \qquad \mbox{for all}~x \in \Omega.$$
 \item For any $x \in A$ and any $C \subset \Omega$, we have
 $$ \mathbb{P}\left(x_{n+1} \in C\big| x_n \in A\right) \geq \varepsilon \rho(C).$$
 \end{enumerate}
We choose
$$ A=\left[\sqrt{d}, \sqrt{d} + \frac{1}{2} \right].$$
 The first step is to prove that when starting in $x$, the expected hitting time $\tau_A$ is finite. 
 We first argue that if $0 \leq \|x_n\| \leq \sqrt{d}$, then there is a positive probability to hit $A$ in the next $2d$ steps. If we do not hit $A$, then either $0 \leq \|x_{n+2d}\| \leq \sqrt{d}$ in which case we can run the argument again or $\|x_{n+2d}\| \geq \sqrt{d} + 1/2$. We will then show, by means of a separate argument, that if $\|x_{n}\| \geq \sqrt{d} + 1/2$, then there is drift back to the origin that leads us back to the regime $\|x_{n+k}\| \leq \sqrt{d}$. 
 Let us first assume that $0 \leq \|x_n\| \leq \sqrt{d}$. We make use of the explicit representation
  $$  \|x_{n+1}\| = \left( \|x_n\|^2 - 2 \cdot \frac{|\gamma_1|}{(\gamma_1^2 + \chi^2(d-1))^{1/2}} \cdot \|x_n\| + 1 \right)^{1/2}.$$
 For each point $0 \leq \|x_n\| \leq \sqrt{d}$ there is a nonzero probability that the next $\sim 2d$ steps have $|\gamma_1| \leq 1/100$ and $\left| \chi^2(d-1) - d \right| \leq \sqrt{d}$ which leads to growth of the sequence. Note that
 $$ \|x_{n+1}\| \leq  \left( \|x_n\|^2 + 1 \right)^{1/2} \leq \|x_n\| + \frac{1}{2\|x_n\|}.$$
 This means that if we take $\sim 2d$ steps, we will surpass $\sqrt{d}$. Moreover, when $\|x_n\| \sim \sqrt{d}$, the jumps are of size $\lesssim d^{-1/2} < 1/2$ and we are sure to hit $A$. Let us now assume that $\| x_n \| > \sqrt{d} + 1/2$. It is clear from the representation and the argument above that then
 $$ \|x_n\| - 1 \leq \|x_{n+1}\| \leq \|x_n\| + \frac{1}{2\|x_n\|}.$$
 We argue that the process is much more likely to decrease. One has
\begin{align*}
 \|x_{n+1}\|^2 - \|x_n\|^2 &= - 2 \cdot \frac{|\gamma_1|}{(\gamma_1^2 + \chi^2(d-1))^{1/2}} \cdot \|x_n\| + 1 \\
 &\leq - 2 \cdot \frac{|\gamma_1|}{(\gamma_1^2 + \chi^2(d-1))^{1/2}} \cdot \sqrt{d} + 1.
 \end{align*}
 Thus, conditioning on $\|x_n\|^2 \geq \sqrt{d}$, we have
 $$ \mathbb{E} \left(  \|x_{n+1}\|^2 - \|x_n\|^2 \right) \leq 1 - 2 \cdot \mathbb{E}~ \frac{|\gamma_1| \sqrt{d}}{(\gamma_1^2 + \chi^2(d-1))^{1/2}}.$$
 In large dimensions, we can use that $\chi^2(d-1) = d \pm \mathcal{O}(\sqrt{d})$ to deduce that
 $$  \mathbb{E}~ \frac{|\gamma_1| \sqrt{d}}{(\gamma_1^2 + \chi^2(d-1))^{1/2}} = \sqrt{\frac{2}{\pi}} + o(1) > \frac{1}{2}$$
 which implies negative drift back to the origin. In small dimensions, one can replace $x_n > \sqrt{d}$ by $x_n > 10000 \sqrt{d}$ and obtain the same estimate with fairly crude bounds on the $\chi^2(d-1)$ distribution. As for the second part of the construction, we use a standard approach in the area and let $\rho$ be a multiple of the Lebesgue measure localized to an interval
 $$ \rho = 1_{\sqrt{d} -1/4 < x < \sqrt{d}} \cdot 4 dx.$$
 The representation 
$$  \|x_{n+1}\| = \left( \|x_n\|^2 - 2 \cdot \frac{|\gamma_1|}{(\gamma_1^2 + \chi^2(d-1))^{1/2}} \cdot \|x_n\| + 1 \right)^{1/2}$$
shows that $\|x_{n+1}\|$ can be anywhere in $[\|x_n\| - 1, \sqrt{\|x_n\|^2 + 1}]$ and it does so with a density that is uniformly bounded away from 0 in any strict subinterval. This shows that if $\sqrt{d} < \|x_n\| < \sqrt{d} +1/2$, then the likelihood of $\|x_{n+1}\|$ being somewhere in $[\sqrt{d} - 1/4, \sqrt{d}]$ is bounded away from 0 with positive probability. This shows that the process is a recurrent Harris chain.  It is easy to see from the definition that it is aperiodic. The Harris chain is recurrent and thus, see \cite[Theorem 10.0.1]{meyn}, it has a unique (up to multiples) invariant measure. Moreover, the argument also shows that it is positively recurrent.
\end{proof}

\subsection{A Lyapunov Lemma.}  The purpose of this section is to establish a simple inequality
for a Lyapunov function. We will then use it in the subsequent section to establish existence of all moments
and fast decay of $\pi_d$.

\begin{lemma}
Let $\alpha > 0$. Then, for some $c_d > 0$ depending only on the dimension
$$ \mathbb{E}~ e^{\alpha \|x_{n+1}\|^2} \leq \frac{ e^{\alpha} c_d}{ 2 \alpha \|x_n\|} \cdot e^{\alpha \|x_{n}\|^2}.$$
\end{lemma}
\begin{proof}
We use
$$  \|x_{n+1}\|^2 =  \|x_n\|^2 - 2 \cdot \frac{|\gamma_1|}{(\gamma_1^2 + \chi^2(d-1))^{1/2}} \cdot \|x_n\| + 1.$$
Rewriting this inequality, we get
\begin{align*}
 \mathbb{E}~ e^{\alpha \|x_{n+1}\|^2} &=  \mathbb{E}~ \exp\left( \alpha \|x_n\|^2 - 2\alpha \cdot \frac{|\gamma_1|}{(\gamma_1^2 + \chi^2(d-1))^{1/2}} \cdot \|x_n\| + \alpha \right) \\
 &= e^{\alpha \|x_n\|^2} ~ e^{\alpha} \cdot\mathbb{E} \exp \left(  -2\alpha\|x_n\| \cdot \frac{|\gamma_1|}{(\gamma_1^2 + \chi^2(d-1))^{1/2}} \right).
\end{align*}
We will try to bound this rather precisely by thinking of 
$$ X = \frac{|\gamma_1|}{(\gamma_1^2 + \chi^2(d-1))^{1/2}}$$
as the absolute value of a coordinate of a randomly chosen point on the sphere. To find the probability density of $X$, we let $A_d(r) = a_d r^d$ denote the surface area of a $d$-sphere with radius $r$, where $a_d \in \mathbb{R}$ is a dimension-dependent constant.  Firstly,
$$A_d(1) = \int_0^\pi A_{d-1}(\sin \theta) d\theta = \int_0^\pi a_{d-1} \sin^{d-1} \theta d\theta$$
Due to the fact there's an absolute value, we integrate solely on the top cap of the sphere, hence from $\arccos x$ to $\pi/2$.
\begin{align*}
    \mathbb{P}_d(X \le x) &= \frac{\int_{\arccos x}^{\frac{\pi}{2}} A_{d-2} (\sin \theta) d\theta}{\int_0^\frac{\pi}{2} A_{d-2} (\sin \theta) d\theta} = \frac{\int_{\arccos x}^{\frac{\pi}{2}} \sin^{d-2} \theta d\theta}{\int_0^{\frac{\pi}{2}} \sin^{d-2} \theta d\theta}\\
    &= \frac{2 \Gamma(\tfrac{d}{2})}{\sqrt{\pi} \Gamma(\frac{d-1}{2})} \int_{\arccos x}^{\frac{\pi}{2}} \sin^{d-2} \theta d\theta = c_d \int_{\arccos x}^{\frac{\pi}{2}} \sin^{d-2} \theta d\theta.
\end{align*}
Thus,
$$p_d(x) = \frac{d}{dx} \mathbb{P}_d(X \le x) = c_d (1 - x^2)^{\frac{d-3}{2}}.$$
Therefore, for any parameter $t > 0$, we have
\begin{align*}
 \mathbb{E} \exp \left(  -t \frac{|\gamma_1|}{(\gamma_1^2 + \chi^2(d-1))^{1/2}} \right) &= \int_0^1 e^{-t z} p_d(z) dz \\
 &\leq c_d\int_0^1 e^{-t z}  dz \\
 &\leq c_d \int_0^{\infty} e^{-t z} dz = \frac{c_d}{t}.
 \end{align*}
 Thus
 $$ \mathbb{E} \exp \left(  -2\alpha\|x_n\| \cdot \frac{|\gamma_1|}{(\gamma_1^2 + \chi^2(d-1))^{1/2}} \right) \leq \frac{c_d}{ 2 \alpha \|x_n\|}.$$
\end{proof}

\subsection{Boundedness of Exponential Moments}
We make use of the following Lemma from Meyn-Tweedie \cite{meyn}. We state it as it appears in \cite{meyn}, they use the abbreviations
$$PV(x) = \int_{\Omega} V(y)P(x, dy) \qquad \mbox{and} \qquad \pi(f) = \int_{0}^{\infty} f(x)\pi(dx) = \mathbb{E}_{\pi}[f(X)].$$
With this notation, the statement reads as follows.
\begin{lemma}[Theorem 14.3.7 \cite{meyn}] Suppose that $\Phi$ is positive recurrent with invariant probability $\pi$, and suppose that $V$, $f$ and $s$ are non-negative, finite-valued functions on $X$ such that
$$ PV(x) \leq V(x) - f(x) + s(x)$$
for every $s \in X$. Then $\pi(f) \leq \pi(s)$.
\end{lemma}

We will now use this to prove that for every $\alpha > 0$
$$ \int_{0}^{\infty} e^{\alpha x^2} d\pi_d(x) < \infty.$$

\begin{proof}
We choose $V(x) = e^{\alpha x^2}$. The Lemma in the previous section shows that
for all $x$ sufficiently large (depending only on the dimension)
$$PV(x)  \leq \frac{ e^{\alpha} c_d}{ 2 \alpha x} \cdot V(x).$$
This estimate is not particularly good when $0 \leq x \leq 1$ in which case, we can always use that
$x_n \leq 2$ implies that $x_{n+1} \leq 3$ and thus $x \leq 2 \implies PV(x)  \leq C.$
Let $z = e^{\alpha} c_d/\alpha$. Then, for $x \geq z$, we have $PV(x) \leq V(x)/2$. Thus, after possibly increasing the constant $c_d$, we get
$$PV(x)  \leq \frac{ e^{\alpha} c_d}{ 2 \alpha \max\left\{x,1\right\}} \cdot V(x).$$
This means that
$$ f(x) = \frac{V(x)}{2} 1_{x > z} \qquad \mbox{and} \qquad s(x) = \frac{ e^{\alpha} c_d}{ 2 \alpha x} \cdot V(x) 1_{x < z}$$
are valid choices to ensure 
$$ PV(x) \leq V(x) - f(x) + s(x)$$
for all $x \geq 0$. The Lemma implies that
$$ \pi_d(f) =  \frac{1}{2} \int_{z}^{\infty} V(x) d \pi_d(x)  \leq \int_{0}^{z}  \frac{ e^{\alpha} c_d}{ 2 \alpha \max\left\{x,1\right\}} V(x) d\pi_d(x) = \pi_d(s) < \infty.$$
However,
$$  \int_{0}^{\infty} e^{\alpha x^2} d\pi_d(x) =  \int_{0}^{z} e^{\alpha x^2} d\pi_d(x) +  \int_{z}^{\infty} e^{\alpha x^2} d\pi_d(x)$$
and the first integral is $\leq e^{\alpha z^2} < \infty$ and the second integral is finite and we have the desired result.
\end{proof}

 \subsection{Computing the mean (informally).}
The decay estimate from the preceding section implies that the second moment exists. We can use this to determine the mean exactly. We first give an informal estimate. Using
 $$  \|x_{n+1}\|^2 =  \|x_n\|^2 - 2 \cdot \frac{|\gamma_1|}{(\gamma_1^2 + \chi^2(d-1))^{1/2}} \cdot \|x_n\| + 1$$
  and iterating this recursion, we deduce
  $$ \|x_{n+k}\|^2 - \|x_n\|^2  =  k - \sum_{\ell=0}^{k-1}  2 \cdot \frac{|\gamma_1|}{(\gamma_1^2 + \chi^2(d-1))^{1/2}} \cdot \|x_n\|.$$
 Taking an average, we have
 $$  \frac{ \|x_{n+k}\|^2 - \|x_n\|^2}{k}  =  1 -  \frac{1}{k}\sum_{\ell=0}^{k-1}  2 \cdot \frac{|\gamma_1|}{(\gamma_1^2 + \chi^2(d-1))^{1/2}} \cdot \|x_n\|.$$
Since the second moment exists, we can take an expectation on both sides and let $k \rightarrow \infty$ to deduce that
$$ 0 = 1 - \mathbb{E} \left( 2 \cdot \frac{|\gamma_1|}{(\gamma_1^2 + \chi^2(d-1))^{1/2}} \right) \mathbb{E}\|x_n\|$$
 which suggests that the first moment is 
 $$ \int_0^{\infty} x ~d\pi_d =  \left( \mathbb{E} \frac{  2 |\gamma_1|}{(\gamma_1^2 + \chi^2(d-1))^{1/2}} \right)^{-1}.$$
 This formula can also be used to deduce the asymptotic behavior: as $d$ increases, we get tight concentration of the $\chi^2-$random variable $ \chi^2(d-1) \sim d \pm \sqrt{d}$
 suggesting 
 $$ \mathbb{E} \frac{  2 |\gamma_1|}{(\gamma_1^2 + \chi^2(d-1))^{1/2}} \sim \mathbb{E} \frac{  2 |\gamma_1|}{\sqrt{d}} = \frac{2}{\sqrt{d}} \sqrt{\frac{2}{\pi}}$$
which suggests that, as $d \rightarrow \infty$,
$$ \mathbb{E} \|x_n\| = \frac{\sqrt{\pi}}{\sqrt{8}} \sqrt{d} + o(\sqrt{d}).$$
 This is indeed the case, we will later determine the mean exactly.

\subsection{Another stochastic representation} As is perhaps unsurprising, we could also try to formulate many of these arguments using an explicit representation of the transition kernel.

\begin{lemma}
The stationary distribution $\pi$ satisfies
    \begin{equation*}
        \pi(y) = \left\{ \begin{array}{ll}
            \int_{1-y}^{y+1} P_d(x, y) \pi (x) dx & \quad \text{if } 0 \le y < 1\\
            \int_{\sqrt{y^2-1}}^{y+1} P_d(x, y) \pi(x) dx & \quad \text{if } y \ge 1
        \end{array} \right.
    \end{equation*}
    where $P_d$ is the transition kernel in $d$ dimensions:
    $$P_d(x, y) = \frac{2 \Gamma(\tfrac{d}{2})}{\sqrt{\pi} \Gamma(\frac{d-1}{2}) } \frac{y}{x} \left(1 - \left( \frac{x^2 - y^2 + 1}{2x} \right)^2\right)^\frac{d-3}{2}.$$
\end{lemma}

\begin{proof}
    By definition of stationarity, $\pi$ must satisfy
    $$\pi(y) = \int_{-\infty}^\infty P_d(x, y) \pi(x) dx$$
    As a brief note on bounds of integration, if $0 \le y < 1$ then we can either step down to it from $y + 1$ or step through the origin from $1 - y$, and if $y \ge 1$ then we can again either step down to it from $y + 1$ or step up to it from $\sqrt{y^2 - 1}$.
    Recall that we can interpret $P_d(Y \le y | X = x)$ geometrically as the surface area ratio of a spherical cap to a hemisphere (in $d - 1$ dimensions). The angle defining this cap is given by the law of cosines to be
    $$\theta(x, y) = \arccos \left( \frac{x^2 - y^2 + 1}{2x} \right)$$
    Using the same construction as we did for Lemma 1, we have
    $$\mathbb{P}_d(Y \le y | X = x) = c_d \int_0^{\theta(x, y)} \sin^{d-2} \theta d\theta$$
    This gives
    $$P_d(x, y) = \frac{d}{dy} P_d(Y \le y | X = x) = c_d \sin^{d-2} \theta(x, y) \frac{\partial}{\partial y} \theta(x, y)$$
    where
    \begin{align*}
           \sin \theta(x, y) &= \left(1 - \left( \frac{x^2 - y^2 + 1}{2x} \right)^2\right)^\frac{1}{2} \\
              \frac{\partial}{\partial y} \theta(x, y) &= \frac{y}{x}\left(1 - \left( \frac{x^2 - y^2 + 1}{2x} \right)^2\right)^{-\frac{1}{2}}
    \end{align*}
    Combining this all together, we obtain
    $$P_d(x, y) = c_d \frac{y}{x} \left(1 - \left( \frac{x^2 - y^2 + 1}{2x} \right)^2\right)^\frac{d-3}{2}.$$
\end{proof}

\textbf{Remark.} There is one particularly nice special case: when $d=3$, the transition kernel is 
$$ P_3(x,y) = \frac{y}{x}$$
and the invariant measure satisfies, for $y \geq 1$,
$$ \pi_3(y)  =  \int_{\sqrt{y^2-1}}^{y+1} \frac{y}{x} \pi_3(x) dx.$$
This can be rewritten as the curious identity
$$\forall~y \geq 1 \qquad \qquad  \frac{\pi_3(y)}{y} = \int_{\sqrt{y^2-1}}^{y+1} \frac{\pi_3(x)}{x} dx.$$
 It is a natural question whether this identity could be used to derive more precise results when $d=3$.

\subsection{Computing the mean (formally).}

As we did with our previous mean calculation, we begin by setting up
$$\|x_{n+1}\|^2 =  \|x_n\|^2 - 2 \cdot \frac{|\gamma_1|}{(\gamma_1^2 + \chi^2(d-1))^{1/2}} \cdot \|x_n\| + 1$$
Taking the expectation we get
$$\mathbb{E}[\|x_{n+1}\|^2] =  \mathbb{E}[\|x_n\|^2] - 2 \mathbb{E}[X] \mathbb{E}[\|x_n\|] + 1$$
where $X$ is shown in the proof for Lemma 1 to satisfy
$$p_d(x) = c_d (1 - x^2)^{\frac{d-3}{2}}$$
Letting $n \to \infty$, since we have stationarity and finite second second moments, we can write
$$\mathbb{E}[\|x_n\|] = \frac{1}{2 \mathbb{E}[X]}$$
We have
$$\mathbb{E}[X] = c_d \int_0^1 x (1 - x^2)^{\frac{d-3}{2}} dx = \frac{c_d}{2} \int_0^1 (1 - u)^{\frac{d-3}{2}} dx = \frac{c_d}{2} \cdot B\left(1, \frac{d-1}{2} \right)$$
where $B(x, y) = \int_0^1 t^{x-1} (1-t)^{y-1} dt$ denotes the Beta function. Thus,
$$\mathbb{E}[X] = \frac{c_d}{2} \cdot \frac{\Gamma(1)\Gamma(\frac{d-1}{2})}{\Gamma(1 + \frac{d-1}{2})} = \frac{c_d}{2} \cdot \frac{\Gamma(\frac{d-1}{2})}{\Gamma(\frac{d+1}{2})} = \frac{\Gamma(\tfrac{d}{2})}{\sqrt{\pi} \Gamma(\frac{d+1}{2})}$$
Putting this all together,
$$\mathbb{E}[\|x_n\|] = \frac{\sqrt{\pi} \Gamma(\frac{d+1}{2})}{2\Gamma(\tfrac{d}{2})}.$$

\section{Van der Corput: base 2}
For this section, $(a_n)_{n \in \mathbb{N}}$ corresponds to the base 2 van der Corput sequence. 

\subsection{Understanding the different regimes} This subsection focuses on proving a set of results that we require for the proof of Theorem~\ref{th:general}. These results will be slightly more general than necessary to prove Theorem~\ref{th:general}, in particular Lemma~\ref{lemma:doublestep}, but provide a complete picture of the behavior of the sequence. We point out to the reader that while the van der Corput sequence is defined from 0 onwards, our process begins in $z_{-1}$. As such, there is an offset of 1 in all the modulo conditions: the first step from $z_{-1}$ to $z_{0}$ uses $\Vdc_2(0)$. We begin by showing that every 2 points, the norm cannot have increased.

\begin{proposition}\label{prop:stdec}
    For all $z_{-1}\in \mathbb{R}^2$, for all $n\in \mathbb{N}$, either the sequence is indeterminate or $|z_{2n-1}|\geq |z_{2n+1}|$, with equality if and only if $z_{2n-1}=z_{2n+1}$.
\end{proposition}

\begin{proof}
    Consider any starting position $z_{-1}\in \mathbb{R}$ and $n\in \mathbb{N}$. By definition of the van der Corput sequence, $\Vdc_2(2n)=\Vdc_2(2n+1)-0.5$. $z_{2n-1}$, $z_{2n}$ and $z_{2n+1}$ are aligned, and the two possible choices for $+$ and $-$ in $z_{2n}$ are exactly $z_{2n-1}$ and its symmetric relative to $z_{2n}$. If the symmetric has strictly smaller norm, it is chosen as $z_{2n+1}$. If it has equal norm, the sequence is no longer determined. If it has strictly higher norm, then $z_{2n+1}=z_{2n-1}$. These three cases prove the desired result.
\end{proof}

To obtain a guarantee on the decrease of the norm, we need to study more than 2 consecutive steps. The following Lemmas study blocks of 4 or 8 steps to cover all possible cases. This relies on the following observation: for $k=0$ modulo $b^j$, $\Vdc_b(k+b^{j-1})-\Vdc_b(k)=1/b^{j}$.

\begin{lemma}\label{lemma:doublestep}
    Let $k \equiv -1$ modulo 4. Then if $z_k$ exists and $|z_k|<\sqrt{2}$, we have at least one of the following:
    \begin{itemize}
        \item[1.] $z_k=z_{k+2}$
        \item[2.] $z_{k+2}=z_{k+4}$.
    \end{itemize}
\end{lemma}

\begin{proof}
    Let $k \equiv -1$ modulo 4. By the definition of the van der Corput sequence, we know $\Vdc_b(k+1)=\Vdc_b(k)+0.5$, $\Vdc_b(k+3)=\Vdc_b(k+2)+0.5$ and $\Vdc_b(k+2)=\Vdc_b(k)+0.25$. In other words, the first two steps will be along the same line $D_1$, and the next two steps along the line $D_2$ orthogonal to $D_1$ and intersecting it in $z_{k+2}$. We will prove the result by contradiction.
    If $z_k \neq z_{k+2}$ and $z_{k+2} \neq z_{k+4}$, then $|z_{k+4}|<|z_{k+2}|<|z_{k}|$ by Proposition~\ref{prop:stdec}. Since none of $z_{k}$, $z_{k+2}$ and $z_{k+4}$ are equal to each other, the steps in each pair of steps were done in the same direction. We have $|z_{k+2}-z_{k}|=2$ and  $|z_{k+4}-z_{k+2}|=2$ and the three points $z_{k}$, $z_{k+2}$ and $z_{k+4}$ form a triangle with $\widehat{z_{k+4}z_{k+2}z_{k}}=\pi/2$ modulo $\pi$.  We therefore have $|z_{k+4}-z_{k}|=2\sqrt{2}$ and, by hypothesis, $|z_{k}|<\sqrt{2}$. This gives $|z_{k+4}|>\sqrt{2}>|z_{k}|$, a contradiction.
\end{proof}

\begin{lemma}\label{lemma:failstep}
    Let $k \equiv -1$ modulo 4. Suppose that $z_k$ exists and that $|z_k|=r>1$. If $z_{k}=z_{k+2}$, then 
    $$|z_k|-|z_{k+4}|>r-\sqrt{r^2+4-4\sqrt{r^2-1}}.$$ Furthermore, if $\sqrt{2}<|z_k|<2$, then $z_k \neq z_{k+4}$ and $|z_{k+4}|<\sqrt{2}$.
\end{lemma}
\begin{proof}
     Let $k \equiv -1$ modulo 4 such that $|z_k|=r>1$ and $z_{k}=z_{k+2}$. Firstly, we suppose without loss of generality that $z_{k}$ is given by $(r,0)$. We are only trying to determine distances, and these are invariant by rotation. As illustrated in Figure~\ref{fig:failstep1}, if $z_{k+2}=z_{k}$, this means that the intersections between the circle $C_z$ of radius 2 centered in $z_k$ and the line $D_1$ defined by the step with $\Vdc_2(k+1)$ are outside the circle $C_o$ of radius $r$ centered in the origin. We also know that the direction $D_2$ of the next two steps is orthogonal to $D_1$. We will now determine the worst norm possible for $z_{k+4}$.

\begin{figure}[h!]
    \centering
    \begin{tikzpicture}

    \draw[blue] (0,0) circle (2.0);
    \filldraw[blue] (0,0) circle (1.pt);
    \node[align=left][blue] at (0.0,-0.3) {$O$};
    \node[align=left][blue] at (-1.8,0.0) {$C_o$};
    \draw[red] (2,0) circle (1.0);
    \filldraw[red] (2,0) circle (1.pt);
    \node[align=left][red] at (2.3,-0.3) {$z_{k}$};
    \node[align=left][red] at (3.2,-0.3) {$C_z$};
    \draw[black] (1.5,1.9364) -- (2,0) -- (2.5,-1.9364);
    \draw[black] (1.5,-1.9364) -- (2,0) -- (2.5,1.9364);
    \node[align=left][black] at (2.0,-1.5) {$C_2$};
    \draw[black] (1.031,0.2588) arc (165:195:1.0);
    \node[align=left][black] at (1.4,0.0) {$CA$};
    \draw[black] (2.0,2.0) -- (2.0,0.0);
    \draw[black] (2.0,2.0) arc (90:105:2);
    \node[align=left][black] at (1.8,1.8) {$\alpha$};
    \end{tikzpicture}
    \caption{Given a failed step in the double cone $C_2$, the next step has to be in the arc $CA$.}
    \label{fig:failstep1}
\end{figure}
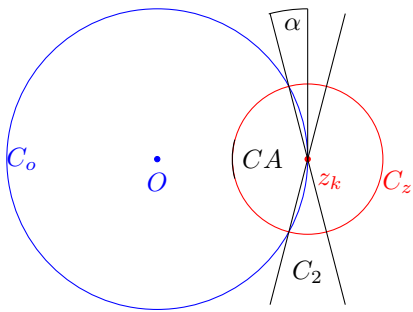
     Let $\alpha$ be the angle between the vertical line $y=r$ and the line passing through $z_{k}$ and the upper intersection of $C_o$ and $C_z$. The directions that would lead to $z_{k}=z_{k+2}$ are in the ``double'' cone $C_2$ of angle $2\alpha$ with summit $z_{k}$. The possible next steps are formed by the circle arc $CA$ defined by the intersection between $C_z$ and the lines orthogonal to those contained in $C_2$. We first determine $\alpha$. Let $(x,y)$ be the upper intersection of $C_o$ and $C_z$. We have $x^2+y^2=r^2$ and $(x-r)^2+y^2=4$. Solving this gives us $x=r-2/r$ and $y=2\sqrt{1-1/r^2}$. We then obtain $\sin(\alpha)=1/r$. The point $p$ on $CA$ furthest from the origin $O$ is such that $\widehat{Oz_kp}=\alpha$. Its $y$-coordinate $p_y$ is such that $\sin(\alpha)=p_y/2=1/r$. As such, $p_y=2/r$. Its $x$-coordinate is $r-2\sqrt{1-1/r^2}$. We can now determine the distance $d$ to the origin,
     $$d=\sqrt{r^2+4-4\sqrt{r^2-1}}.$$
     Among all the points on the arc $CA$, this is the furthest away from the origin, as $CA$ is symmetric relative to the $x$-axis. Finally, we obtain that $|z_{k}|-|z_{k+4}|>r-\sqrt{r^2+4-4\sqrt{r^2-1}}$, which is the first part of the desired result.
     We now suppose that $\sqrt{2}<|z_k|<2$. Since $z_k=z_{k+2}$, the origin has to be in a band of width 1 orthogonal to the direction $D_1$ imposed by $\Vdc_2(k+1)$. Let $D_2$ be the direction of this band. Since $k \equiv -1$ modulo 4, the next two steps of the process are along this direction $D_2$. If $z_{k+4}=z_{k}$, then the origin must be strictly inside a square of side 1 with one corner equal to $z_{k}$. This would mean $|z_{k}|<\sqrt{2}$, a contradiction. As such, $z_{k+4} \neq z_k$. In this case, since $|z_{k}|<2$, the origin has to be in a rectangle $R$ of side $[z_k, z_{k+4}]$ of length 2 and whose adjacent side is of length 1. Since $|z_{k}|>\sqrt{2}$, we also know that it will be strictly inside the square $S_2$ of size 1 that forms half of $R$ and has $z_{k+4}$ as a summit. We conclude $|z_{k+4}|<\sqrt{2}$.
\end{proof}

\begin{lemma}
     Let $k \equiv -1$ modulo 8. If $z_k$ exists and $z_{k}$, $z_{k+2}$ and $z_{k+4}$ are all distinct, then $|z_k|-|z_{k+8}|\geq 2(\sqrt{2}-1)$ if $|z_{k}|>2$, and $|z_{k+8}|<1+(\sqrt{2}-1)^2$ otherwise.
\end{lemma}

\begin{proof}
     Let $k \equiv -1$ modulo 8, such that $z_{k}$, $z_{k+2}$ and $z_{k+4}$ are all distinct. By contrapositive of Lemma~\ref{lemma:doublestep}, we have  $|z_k|>\sqrt{2}$. The reasoning used in the proof of Lemma~\ref{lemma:doublestep} relative to the positioning of $z_{k}$, $z_{k+2}$ and $z_{k+4}$ still holds: they form a triangle with right-angle for $\widehat{z_{k+4}z_{k+2}z_{k}}$. Furthermore, the region where the origin can be is a cone $C$ with summit in the middle of $[z_{k},z_{k+4}]$ and with sides parallel to, and with same directions as, $(z_{k},z_{k+2})$ and $(z_{k+2},z_{k+4})$. An illustration is provided in Figure~\ref{fig:allsteps}.
 We now use the van der Corput sequence's definition: if $k \equiv -1$ modulo 8, $\Vdc_2(k+5)=\Vdc_2(k)+1/8$. There are now two possibilities, depending on whether or not $z_{k}$, $z_{k+4}$, $z_{k+5}$ and $z_{k+6}$ are aligned. We will treat the aligned case in detail, the other works in a very similar way. Let $D$ be that line. By Proposition~\ref{prop:stdec}, $|z_{k+6}|\leq|z_{k+4}|$, with equality iff the two points are equal.\\
     We distinguish three cases. If $z_{k+6}=z_{k+4}$, then the origin is in an open band of width 1 orthogonal to $(z_{k+4},z_{k})$ where one of the two side-edges passes through $z_{k+4}$. We also had previously that the origin is in the cone $C$. The point in this convex set that minimizes $|z_{k}|-|z_{k+6}|$ is on $[z_{k},z_{k+6}]$, and by construction is at distance 1 of $z_{k+6}$. We also know that $|z_{k+6}-z_{k}|=\sqrt{8}$. For this point we have $|z_{k}|-|z_{k+6}|=2(\sqrt{2}-1)$.

    \begin{figure}[h!]
    \centering
    \begin{tikzpicture}[scale=1.1]
    \draw[black] (0,0) --(1,0) -- (2,0) --(2,1) --(2,2);
    \draw[blue] (0,0) circle (2.pt);
    \draw[blue] (1,0) circle (2.pt);
    \draw[blue] (2,0) circle (2.pt);
    \draw[blue] (2,1) circle (2.pt);
    \draw[blue] (2,2) circle (2.pt);
    \draw[blue,dashed] (0,1.5)--(0,-0.5);
    \draw[blue,->] (0,1)--(0.3,1);
    \draw[blue,dashed] (1,1.5)--(1,-0.5);
    \draw[blue,->] (1,1)--(1.3,1);
    \draw[blue,dashed] (1.5,0)--(3.5,0);
    \draw[blue,->] (2.3,0)--(2.3,0.3);
    \draw[blue,dashed] (0.5,1)--(3.5,1);
    \draw[blue,->] (2.3,1)--(2.3,1.3);
    \draw[black,->] (0,0)-- (3.6,3.6);
    \node[align=left][black] at (3.6,3.2) {$D$};
    \node[align=left][black] at (0.2,-0.2) {$z_k$};
    \node[align=left][black] at (1.5,-0.2) {$z_{k+1}$};
    \node[align=left][black] at (2.5,-0.2) {$z_{k+2}$};
    \node[align=left][black] at (2.5,0.8) {$z_{k+3}$};
    \node[align=left][black] at (2.5,1.8) {$z_{k+4}$};
    \draw[green] (0.586,0.586) circle (2.pt);
    \draw[green] (1.293,1.293) circle (2.pt);
    \draw[red] (3.414,3.414) circle (2.pt);
    \draw[red] (2.707,2.707) circle (2.pt);
    \draw[draw=yellow,fill=yellow,opacity=0.2]
        (1.0, 1.0)--(3.6, 1.0)--(3.6 ,3.6)--(1.0, 3.6)--cycle;
    \draw[black,dashed] (1.0,1.0) --(3.6,1.0);
    \draw[black, dashed] (1.0,1.0) --(1.0,3.6);
    \draw[blue,dashed] (0.793,1.793)--(1.793,0.793);
    \end{tikzpicture}
    \caption{The possible positions of $z_{k+5}$ and $z_{k+6}$ are in red and green, and we could have $z_{k+6}=z_{k+4}$. In yellow, the cone $C$ where the origin could be based on the first 4 steps. The non-dashed edges indicate the sides where the region is infinite.}
    \label{fig:allsteps}
\end{figure}
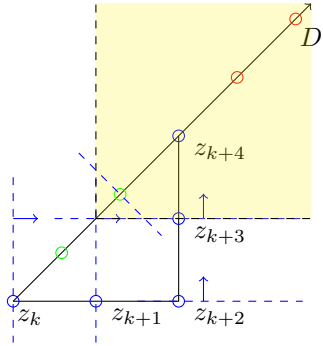
     
     In the second case, $z_{k+6}$ is on $(z_{k+4},z_{k}]$ but not between $z_{k}$ and $z_{k+4}$. The acceptable region for the origin is the intersection of $C$ and the halfplane defined by the bisector of $[z_{k+6},z_{k+4}]$. Since $|z_{k+4}-z_{k}|=\sqrt{8}$ and $z_{k}$ is to the right of $z_{k+4}$ on $D$, we have that $|z_{k}|-|z_{k+6}|\geq 1$. In the final case, $z_{k+6}$ is on $D$ and between $z_{k}$ and $z_{k+4}$. The possible positions for the origin are all in the triangle formed by the top of the cone $C$, cut off at the halfplane orthogonal to $D$ and at distance 1 to $z_{k+4}$. We necessarily have $|z_k|<2$. We now need to find the smallest $|z_{k}|-|z_{k+6}|$ possible. If the origin is on one of the outer edges of the cone,  
     $$|z_{k}|=\sqrt{2+2\sqrt{2}a+2a^2}$$ and 
     $$|z_{k+6}|=\sqrt{6-4\sqrt{2}+2(2-\sqrt{2})a+2a^2},$$ for $a \in [0,\sqrt{2}-1)$. The minimal $|z_{k}|-|z_{k+6}|$ is obtained for $a=\sqrt{2}-1$ as the derivative is always negative on $[0,2-\sqrt{2}]$. This corresponds to $|z_{k+6}|<1+(\sqrt{2}-1)^2<\sqrt{2}$. Since this third case can only happen if $|z_k|<2$, we get the desired result in this case too.
     Grouping all three cases together, we obtain that $|z_{k}|-|z_{k+6}|>2(\sqrt{2}-1)$ if $|z_k|\geq 2$ when $k\equiv -1$ modulo 8, and $|z_{k+6}|<1+(\sqrt{2}-1)^2$ if $|z_k|<2$. Proposition~\ref{prop:stdec} gives the desired result in the aligned case.
     When they are not aligned, if $z_{k+4}=z_{k+6}$, then $z_{k},z_{k+4}, z_{k+7}$ and $z_{k+8}$ are aligned and we use the reasoning above. If $z_{k+4} \neq z_{k+6}$, the origin has to be in cones pointing away from the origin that are much closer to the line containing $z_{k+4}$ and $z_{k+6}$. We omit the calculation details here, they represent a better situation than that proved in the next lemma.
\end{proof}

\begin{lemma}\label{lemma:yesno}
    Let $k \equiv -1$ modulo 8. If $z_k$ exists, $z_{k+2} \neq z_k$, and $z_{k+4}=z_{k+2}$, then:
    \begin{itemize}
        \item[1.] If $|z_k|<2$ then $|z_{k+4}|<\sqrt{2}$.
        \item[2.] If $2\leq |z_k|\leq \sqrt{5}$ then $|z_{k+4}|< \sqrt{8(1-\sqrt{3}/2)}$.
        \item[3.] If $|z_k|>\sqrt{5}$ then $|z_k|>|z_{k+4}|+\sqrt{5}-1$.
    \end{itemize}
\end{lemma}

\begin{proof}
Let $k\equiv- 1$ modulo 8. We suppose that we have $z_{k+2} \neq z_k$ and $z_{k+2}=z_{k+4}$. Since $k\equiv 1$ modulo 8, the third and fourth steps are on a line orthogonal to the one the first two steps are on, and they intersect in $z_{k+2}$. Since we are back in $z_{k+2}$ after the fourth step, this implies that the origin is in a band $B$ of width 1, and one of its sides is defined by $(z_{k},z_{k+2})$. 
Firstly, if $|z_k|<2$, and $z_{k+2}\neq {z_k}$, then $|z_{k}|>1$. We also have that the origin is in a square of side 1, $S_1$, and one of the corners of the square is $z_{k+2}$, see Figure~\ref{fig:yesno} (whether we are to the left or right of $(z_{k},z_{k+2})$ corresponds exactly to the side $z_{k+3}$ is on). As such $|z_{k+4}|<\sqrt{2}$.\\
If $2<|z_{k}|<\sqrt{5}$, then the origin has to be within $(B(z_k,\sqrt{5})\backslash B(z_k,2)) \cap B$. The point furthest away from $z_{k+2}$ in that case is on the outer edge of $S_1$. We clearly have $|z_{k+4}|<\sqrt{2}$, but this can be further improved by considering this point when $|z_k|=2$. In this case, we have 
$$|z_{k+4}|=\sqrt{8(1-\sqrt{3}/2)}.$$
The intersection point between the circle $C(z_k,|z_k|)$ and $B$ will only be closer to $z_{k+4}$ for higher norms of $z_{k}$.

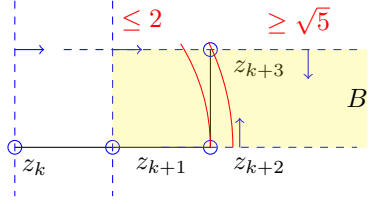
\begin{figure}[h!]
    \centering
    \begin{tikzpicture}[scale=1.3]
        \draw[black] (0,0) --(1,0) -- (2,0)--(2,1);
        \node[align=left][black] at (0.2,-0.2) {$z_k$};
    \node[align=left][black] at (1.5,-0.2) {$z_{k+1}$};
    \node[align=left][black] at (2.5,-0.2) {$z_{k+2}$};
    \node[align=left][black] at (2.5,0.8) {$z_{k+3}$};
    \draw[blue] (0,0) circle (2.pt);
    \draw[blue] (1,0) circle (2.pt);
    \draw[blue] (2,0) circle (2.pt);
    \draw[blue] (2,1) circle (2.pt);
    \draw[blue,dashed] (0,1.5)--(0,-0.5);
    \draw[blue,->] (0,1)--(0.3,1);
    \draw[blue,dashed] (1,1.5)--(1,-0.5);
    \draw[blue,->] (1,1)--(1.3,1);
    \draw[blue,dashed] (1.5,0)--(3.5,0);
    \draw[blue,->] (2.3,0)--(2.3,0.3);
    \draw[blue,dashed] (0.5,1)--(3.5,1);
    \draw[blue,->] (3.0,1)--(3.0,0.7);
    \draw[draw=yellow,fill=yellow,opacity=0.2]
        (1.0, 1.0)--(1.0, 0.0)--(3.6 ,0.0)--(3.6, 1.0)--cycle;
        \node[align=left][black] at (3.5,0.5) {$B$};
    \draw[red] (2.0,0.0) arc (0:32:2.0);
    \node[align=left][red] at (1.3,1.3) {$\leq 2$};
    \node[align=left][red] at (2.9,1.3) {$\geq \sqrt{5}$};
    \draw[red] (2.23,0.0) arc (0:26:2.45);
    \end{tikzpicture}
    \caption{Without loss of generality, one of two acceptable regions for the origin after 4 steps. The other is its symmetric relative to $(z_k,z_{k+2})$. The origin has to be in $B$, and depending on $|z_{k}|$, it is in one of the three regions defined by the red arcs.}
    \label{fig:yesno}
\end{figure}

Finally, if $|z_{k}|>\sqrt{5}$, the origin will still be in $B$ but all we can guarantee is that, when splitting the band in 2 with the orthogonal line to $B$ passing through $z_{k+4}$, it is not in the same part as $z_{k}$. For a given distance $|z_k|$, the furthest point from $z_{k+4}$ is once again the intersection of the circle $C(z_k,|z_{k}|)$ and one of the outer boundary of $B$.
In this case, we obtain 
$$|z_{k+4}|=\sqrt{4+|z|^2-4\sqrt{1+|z|^2}}.$$
The value $|z_k|-|z_{k+4}|$ is strictly increasing in $|z_k|$, and is minimized for $|z_k|=\sqrt{5}$. This corresponds exactly to $|z_{k+4}|=1$ (and $z_{k+3}$ is equal to the origin, as such we require $|z_{k}|>\sqrt{5}$). This gives us the desired result, and concludes the proof.
\end{proof}

This concludes the proof of the results we require to prove Theorem~\ref{th:general}.

\subsection{Proof of Theorem~\ref{th:general}}

\begin{proof}[Proof of Theorem~\ref{th:general}]
    Let $z_{-1} \in \mathbb{R}$. Firstly, we set aside the indeterminate case over which we have no control. Suppose that $(z_n)$ is defined at least for $8|z_{-1}|/(2(\sqrt{2}-1))$ steps. We will show that for any $k \equiv -1$ modulo 8, $|z_{k}|>|z_{k+8}|+2(\sqrt{2}-1)$. For such a $k$, we first consider the case $|z_k|>2$. There are three possibilities:
    \begin{itemize}
        \item If $z_k=z_{k+2}$, then by Lemma~\ref{lemma:failstep}, $|z_{k}|-|z_{k+4}|>r-\sqrt{r^2+4-4\sqrt{r^2-1}}>2(\sqrt{2}-1)$ since $r>2$.
        \item If $z_k \neq z_{k+2}$ and $z_{k+2}=z_{k+4}$ then, by Lemma~\ref{lemma:yesno}, if $|z_{k}|<\sqrt{5}$ then $|z_{k+6}|<\sqrt{2}$ and $|z_{k+6}|<|z_k|-(\sqrt{5}-1)$ otherwise.
        \item If $z_k$, $z_{k+2}$ and $z_{k+4}$ are all distinct then Lemma~\ref{lemma:doublestep} directly states $|z_{k}|>|z_{k+8}|+2(\sqrt{2}-1)$.
    \end{itemize}

    Using Proposition~\ref{prop:stdec}, we have that $|z_{k+4}|$ and $|z_{k+6}|$ are greater than or equal to $|z_{k+8}|$. As such, while $|z_k|>2$, we have the guarantee that $|z_{k}|>|z_{k+8}|+2(\sqrt{2}-1)$. It therefore takes at most $8(|z_{-1}|-2)$ steps to reach the ball of radius 2. With the special cases for $|z_{k}|<2$ from Lemmas~\ref{lemma:doublestep}, \ref{lemma:failstep} and~\ref{lemma:yesno}, it takes exactly one more block of 8 steps to reach $B(0,\sqrt{2})$, proving the desired result.
    
\end{proof}

\subsection{Periodicity}
\begin{proof}[Proof of Theorem~\ref{th:vdc2}]
    We want to show that, for $n$ odd and $z_n \in B(0,1)$, $z_{n+2j}=z_{n}$ and $z_{n+2j+1}$ is on a half-circle centered in $z_{n}$ of radius 1. This half circle is exactly that contained in the halfplane defined by the line orthogonal to $((0,0),z_{n})$ and that contains the origin. We will show this result holds once $z_n$ for $n$ odd is in $B(0,1)$ and this will directly prove the first statement on $B(0,1)\backslash \{(0,0)\}$.
    We first prove that the result is true for the first two steps, i.e. $j=0$ for the first even and odd point after $z_{n}$.

\textit{Even case.} We have that $z_{n+1}=z_n \pm e^{2\pi i \Vdc_2(n+1)}$. By definition, $z_{n+1}$ is on the circle $C_n$ of radius 1 centered in $z_n$, and is on the line $D_n$ defined by $z_n+xe^{2\pi i \Vdc_2(n+1)}$. There are exactly two such points since $D_n$ contains the center of $C_n$. In the first case, if $D_n$ is orthogonal to $((0,0),z_n)$, both points have same norm. The sequence is indeterminate and stops. If not, one of the two elements has norm strictly smaller than the other and is selected. Since any $D_n$ such that the sequence is properly defined cannot be equal to $((0,0),z_n)^{\perp}$,  the element with smaller norm is in the desired halfplane.

\textit{Odd case.} By definition, $z_{n+2}=z_{n+1} \pm e^{2\pi i \Vdc_2(n+2)}$. Since $n$ is even, $\Vdc_2(n+2)=0.5+\Vdc_2(n+1)$. This gives us $z_{n+2}=z_{n+1} \pm e^{2\pi i (\Vdc_2(n+1)+0.5)}=z_{n+1}\pm e^{2\pi i \Vdc_2(n+1)}$. $z_{n+2}$ is therefore on a circle of radius 1 centered on $z_{n+1}$ and on $D_n$ since $z_{n+1}$ is on it. There are two such points separated by a distance of 2, one of which is $z_n$. Since only one of the two can have norm smaller than 1, it will necessarily be selected as $z_{n+2}$. By hypothesis, $|z_n|<1$, therefore $z_{n+2}=z_n$.

By induction, this can be generalized to any $j \in \mathbb{N}$.
\end{proof}

This theorem proves two different elements: firstly any $z_{-1} \in B(0,1)$ is a periodic start position in base 2. Secondly, once we have an odd-numbered $z_{k}\in B(0,1)$ then it will act as a periodic start position for the rest of the sequence. Gaps in the semi-circle then correspond exactly to previously taken directions since no two points of the van der Corput sequence are equal.

\subsection{The uncertain regime: 1 to $\sqrt{2}$}
Theorem~\ref{th:vdc2} shows that once an odd point is in $B(0,1)$, we have a perfect understanding of the future points. Theorem~\ref{th:general} also allows us to understand the behavior of the sequence as long as $|z_n|>\sqrt{2}$. However, apart from the limited information given by Lemma~\ref{lemma:doublestep}, none of the results presented so far give information on what can happen between $1$ and $\sqrt{2}$. Studying bigger blocks of consecutive steps could improve the different lemmas shown previously. However, the following Proposition shows that it is impossible to guarantee that we will reach $B(0,1)$ for any finite number of steps $n$. 

\begin{proposition}\label{prop:nomove}
    For any $n\in \mathbb{N}$, there exist convex sets of starting positions in $B(0,\alpha_n)\backslash B(0,1)$ such that $z_{-1}=z_{2n-1}$, where $\alpha_n>1$ is some constant depending on $n$ such that $\lim_{n\rightarrow \infty} \alpha_n=1$.
\end{proposition}

\begin{proof}
Let $n \in \mathbb{N}$ and let $k=\lfloor \log_2(n) \rfloor$. When considering the first $2^{k+1}$ elements of the sequence, the smallest difference between any two elements is $2^{-k-1}$. Given a point $z_{k} \notin B(0,\sqrt{2})$, we have $z_{k} \neq z_{k+2}$ if and only if $C(z_{k},2)$, the circle centered in $z_k$ of radius 2, has a non-empty intersection with $B(0,|z_{k}|)$. When $|z_k|=1$, this intersection is a single point. As $|z_{k}|$ increases, the circle arc forming $C(z_{k},2)\cap B(0,|z_k|)$ grows. We now need to find positions for which this circle arc corresponds to an angle of less than $2\pi2^{-k-1}$ in $z_{k}$. 
Without loss of generality, we suppose that $z_{k}$ is on the $x$-axis (starting there is obviously incorrect as it would improve in the first two steps). The edge of the circle arc, $z_{k}$ and the origin form a triangle with known distances. Using the law of cosines, we have that $\alpha/2=\arccos((3+r^2)/(4r))$ and we want $\alpha \leq 2\pi 2^{-k-1}$. In the case of equality, this gives us $$3+r^2-4\cos(\pi2^{-k-1})r=0.$$
This has two real roots for $k>1$, both of which are above 1. The smallest of the two is the desired solution. We have 
$$\alpha_n=\frac{4\cos(\pi2^{-k-1})-\sqrt{16\cos^2(\pi2^{-k-1})-12}}{2}$$

Any point on $[O,z_{k}]$ within that distance of the origin will be such that $z_{-1}=z_{2*i-1}$ for $i\leq n$. With a small rotation around the origin to guarantee that none of the first $n$ pairs of steps will improve on $|z_{-1}|$, this proves the desired result.
\end{proof}

The precise regions are larger than simply a line, and their number depends on $\log(n)$. We leave their computation to an interested reader. This result suffices to show that there cannot be a guaranteed improvement for the norm of $|z_k|$ within a fixed number of steps, regardless of that number. This concludes our analysis of the base 2 case. 

\section{Higher bases}
\subsection{General behavior: a discussion}
Before proving Theorem~\ref{th:polygon}, we give some quick elements to explain the behavior of the sequence in higher dimensions. There are two main differences with base 2.
The first is that we no longer have an equivalent to Proposition~\ref{prop:stdec}: we cannot guarantee that every $j$ steps we have a strictly decreasing norm, for some well-chosen value of $j$. With a similar case decomposition as in Lemmas~\ref{lemma:doublestep},~\ref{lemma:failstep} and~\ref{lemma:yesno}, one can show that in base 3 we can lower-bound $|z_{k}|-|z_{k+3}|$ by a constant for $k\equiv -1$ modulo 3 as long as $|z_{k}|>\sqrt{1.5}+c$ where $c$ is a small constant. The $\sqrt{1.5}$, improving on the $\sqrt{2}$ in base 2, corresponds to the maximum distance between the origin and $z_{k}$, while being equidistant to $z_{k+3}$.
This improvement is a side effect of the second difference: we have access to more information on consecutive steps, and we don't have consecutive colinear steps. 
To obtain $\sqrt{1.5}$ as the threshold where we can no longer guarantee improvement, it suffices to consider all the possible choices in 3 consecutive steps, whereas $2^2$ were necessary in the base 2 case to get $\sqrt{2}$ (and even $2^3$ to obtain precise values). As such, $z_{k}$ and $z_{k+3}$ cannot be too far away from each other, and in particular the origin has to be closer to $z_{k+3}$ if it is outside of a small ball centered on $z_k$. For the four different cases to consider here, the maximum distance $|z_k|$ such that $z_k$ and $z_{k+3}$ are equidistant to the origin is respectively 1, $2/\sqrt{3}$ and $\sqrt{1.5}$. 
The size of the ball should be seen as a balance between the angle of consecutive steps and the information their greater number provides. As $b$ increases, Theorem~\ref{th:polygon} shows that the minimal size of the ball we are guaranteed to reach also increases. We do not present calculations for higher bases here, as they provide relatively little information while requiring a case distinction of $2^{b-1}$ cases. 

\subsection{The polygonal structure}

Base $2$ presents a specific behavior as the regular polygon defined by 2 points, if one accepts it as such, is simply a line. For higher bases, we will show that the repeating pattern formed by $b$ consecutive steps forms a regular $b$-sided polygon, which we will abbreviate to $b$-gon. We begin by giving a characterization of all periodic starts for $b>2$. Our results will be shown with the hypothesis that the steps are associated with ``-'' signs. By central symmetry of the sequential construction, any valid $z_{-1}$ obtained for a periodic start with ``-'' will be associated to $-z_{-1}$, a valid periodic start for ``+''.

\begin{proof}[Proof of Theorem~\ref{th:polygonStruct}]
    Without loss of generality, suppose that the periodic sequence $z_n$ is associated with $-$ signs. Let $n\equiv j$ modulo $b$. Firstly, if $j=-1$, then $z_n=z_{-1}$ by definition of a periodic start and they are indeed on the circle of radius 0. We now suppose $j \neq -1$, and let $b_0$ be the closest multiple of $b$ smaller than $n$. Using the periodicity of the sequence, we have 
    \begin{align*}
        z_{n,x}&=z_{n-1,x}-\cos(2\pi \Vdc_b(n))\\
               &=z_{b_0-1,x}-\sum_{i=0}^{j} \cos(2\pi \Vdc_b(b_0+i))\\
               &= z_{-1,x}-\sum_{i=0}^{j} \cos(2\pi \Vdc_b(b_0)+2i\pi/b)\\
               &=z_{-1,x}-\frac{\sin(j\pi/b)}{\sin(\pi/b)}\cos(2\pi \Vdc_b(b_0)+j\pi/b)
    \end{align*}

Similarly, we can obtain for the $y$-coordinate
$$z_{n,y}=z_{-1,y}-\frac{\sin(j\pi/b)}{\sin(\pi/b)}\sin(2\pi \Vdc_b(b_0)+j\pi/b)$$.
This gives us 
\begin{align*}
    |z_n-z_{-1}|&=\sqrt{(z_{n,x}-z_{-1,x})^2+(z_{n,y}-z_{-1,y})^2}\\
    &=\frac{\sin(j\pi/b)}{\sin(\pi/b)}\sqrt{\cos(2\pi V_b(b_0)+j\pi/b)^2+\sin(2\pi V_b(b_0)+j\pi/b)^2}\\
    &=\frac{\sin(j\pi/b)}{\sin(\pi/b)}
\end{align*}
We obtain the first part of the desired result. Furthermore, for $k\equiv -1$ modulo $b$ and $0\leq j \leq b-1$, $\Vdc_b(k+j)=\Vdc_b(j)+\Vdc_b(k)$ and, for $j>0$, $\Vdc_b(k+j+1)-\Vdc_b(k+j)=1/b$. Since we are supposing that all signs are picked to be ``-'', the $b$ steps associated to $\Vdc_b(k),\ldots,\Vdc_b(k+b)$ are exactly the same as $\Vdc_b(0),\ldots,\Vdc_b(b)$ with an extra $\Vdc_b(k)$ rotation. This rotation is the same for all of the steps, therefore both sequences of steps describe the same polygonal shape. The second relation imposes that the angle between two consecutive steps is exactly $2\pi/b$ for each of the $b-1$ steps after the first one. We therefore have that $z_{k}=z_{k+b}$, that each side length is 1, and that each other angle of the $b$-sided polygon is $2\pi/b$. Each sequence of $b$ steps starting in $z_{k}$ for $k \equiv -1$ mod $b$ must be a $b$-gon. Finally, the angle between $(z_{-1},z_{0})$ and $(z_{k},z_{k+1})$ is $\Vdc_b(k)$, this concludes the proof.
\end{proof}

Each $b$-gon corresponding to $b$ consecutive steps will be a rotated copy of the points $(z_{-1},\ldots,z_{b-1})$, with a fixed point in $z_{-1}$. For each copy, the angle of that rotation is given by the first element of the van der Corput sequence in that block of $b$ elements: the $\Vdc_b(bj-1)$ for all $j$ determine the rotations. Since these values are all in $[0,1/b)$, the maximum rotation possible is of angle $2\pi/b$. Figure~\ref{fig:circles} shows these arcs and the circles they lie on. For a rotation of angle $\alpha$, we will write $P_{\alpha}$ for the associated polygon.
\begin{figure}[h!]
    \centering
    \includegraphics[width=0.48\linewidth]{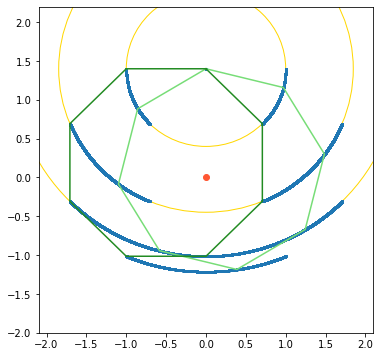}
    \includegraphics[width=0.48\linewidth]{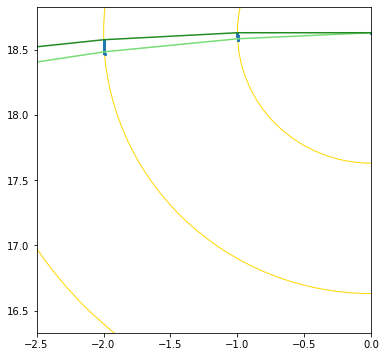}
    \caption{On the left, the first $200\,000$ points of the base $8$ Van der Corput sequence with periodic start, with supporting circles. The first and $102^{nd}$ polygons are drawn. The arcs correspond exactly to those of Figure~\ref{fig:examplesvdc}. On the right, the zoomed in points in first, second and third position of the 117-gon for the base 117 van der Corput sequence, again for $200\,000$ points.}
    \label{fig:circles}
\end{figure}
For such an infinite sequence to be valid with our construction methods, it requires that the norm associated with $z_n-e^{2\pi i V_b(n+1)}$ is smaller than that of $z_n+e^{2\pi i V_b(n+1)}$ for all $n$ (the role of the signs could be reversed, we study this case without loss of generality). This means that the scalar product of $(z_n,(0,0))$ and $(z_n,z_{n+1})$ is strictly positive.
At each point in the sequence, $(0,0)$ must be in the open halfplane containing all the $x$ such that the scalar product of $(z_n,x)$ and $(z_{n},z_{n+1})$ is strictly positive to keep the polygonal structure. In other words, we're considering the halfplane defined by the orthogonal line to the next side of the $b$-gon that passes through $z_n$. For each $b$-gon, we are interested in the intersection of these halfplanes, that forms another $b$-gon for any regular $b$-gon (smaller for $b\geq 5$, equal for $b=4$ and greater for $b=3)$. The origin must be inside this smaller $b$-gon, for all possible rotations. For a polygon $P$, we will denote $P_{i}$ its inner polygon obtained this way.

\begin{lemma}\label{lemma:triangle}
    The intersection of $P_{0,i}$ and $P_{2\pi/b,i}$ is a triangle.
\end{lemma}

\begin{proof}
     We will show that the intersection of $P_{0,i}$ and $P_{2\pi/b,i}$ is exactly a triangle formed by the leftmost submit of $P_{2\pi/b,i}$ cutoff by a vertical edge. We will call $p_{left}$ this summit. Since this structure is invariant by translation, we will consider that $z_{-1,x}=0$. By construction, we know that the rightmost side of $P_{0,i}$ lies on $x=0$. We also have that $P_{2\pi/b,i}$ is a copy of $P_{0,i}$ translated by 1 to the right. Finally, we also know that the inner polygon's axis of symmetry is horizontal as the rightmost side is vertical. As such, it suffices to show that exactly one summit of $P_{2\pi/b,i}$ has its x-coordinate strictly negative and that its adjacent edges both intersect the vertical edge of $P_{0,i}$.
The center $c_i$ of $P_{0,i}$ is on the $y=0$ line by symmetry. We can compute its distance from the right vertical edge, it is given by
$$-c_{i,x}=\frac{\sec(\pi/2-2\pi/b)-\tan(\pi/2-2\pi/b)}{2\tan(\pi/b)}.$$
This is equal to 
$$-c_{i,x}=\frac{(1-\cos(2\pi/b))\cos(\pi/b)}{2\sin(\pi/b)\sin(2\pi/b)}.$$
Expressing all the terms as functions of $\sin(\pi/b)$ and $\cos(\pi/b)$ gives us $c_{i,x}=-0.5$. 
Translated by 1 to the right, the center of the circumscribed circle of $P_{2\pi/b,i}$ is in $0.5$, as such the $x$-coordinate of $p_{left}$ is 
$$0.5-\frac{\sec(\pi/2-2\pi/b)-\tan(\pi/2-2\pi/b)}{2\sin(\pi/b)}.$$
\begin{figure}[h!]

    \centering
    \begin{tikzpicture}
    \def\s{2}
    \pgfmathsetmacro{\R}{\s / (2 * sin(36))}

    \def\startAngle{54}

    \foreach \i in {0,...,4} {
        \pgfmathsetmacro{\Ai}{\startAngle + 72*\i}
        \coordinate (P\i) at (\Ai:\R cm);
    }

    \draw[red, thick] (P0) -- (P1) -- (P2) -- (P3) -- (P4) -- cycle;

    \foreach \i in {0,...,4} {
        \pgfmathsetmacro{\Ai}{\startAngle + 72*\i}
        \pgfmathsetmacro{\thetaPerp}{\Ai + 126 + 90}
        \draw[red, thick] (P\i) -- +(\thetaPerp:{2.1cm});
    }
    \filldraw[black] (P0) circle (2pt);
    \node[above right] at (P0) {$z_{-1}$};

    \begin{scope}[xshift=2cm]
        \foreach \i in {0,...,4} {
            \pgfmathsetmacro{\Ai}{\startAngle + 72*\i}
            \coordinate (Q\i) at (\Ai:\R cm);
        }

        \draw[blue, thick] (Q0) -- (Q1) -- (Q2) -- (Q3) -- (Q4) -- cycle;

        \foreach \i in {0,...,4} {
            \pgfmathsetmacro{\ipone}{int(mod(\i+1,5))}
            \path (Q\i); \pgfgetlastxy{\xA}{\yA}
            \path (Q\ipone); \pgfgetlastxy{\xB}{\yB}
            \pgfmathsetmacro{\thetaEdge}{atan2(\yB - \yA, \xB - \xA)}
            \pgfmathsetmacro{\thetaPerp}{\thetaEdge + 90}
            \draw[blue, thick] (Q\i) -- +(\thetaPerp:{2.1cm});
        }
    \end{scope}
\end{tikzpicture}
\label{fig:oddgons}
\caption{For $b=5$, $P_{0,i}$ in red and $P_{2\pi/b}$ in blue, with their associated inner pentagons. The leftmost vertex of the inner blue pentagon is $p_{left}$.}
\end{figure}
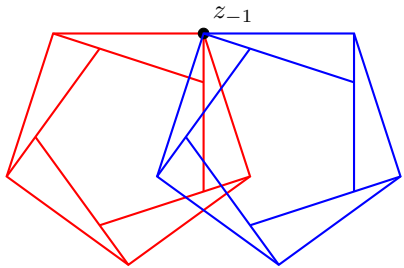

This can be reformulated to 

$$0.5-\frac{\sin(\pi/b)}{\sin(2\pi/b)}=0.5-\frac{1}{2\cos(\pi/b)}.$$
This is strictly negative for all integers $b\geq 5$.
The $x$-coordinate of the adjacent summits is given by $0.5-\cos(2\pi/b)/(2\cos(\pi)/b) $. This is equal to $0.5-\cos(\pi/b)/2+\sin^2(\pi/b)/(2\cos(\pi/b))$, which is positive for $b\geq 5$.

Since $P_{2\pi/b,i}$ is a convex $b$-gon, we have that $P_{2\pi/b,i}$ has a single summit with negative $x$-coordinate. Since this summit is aligned with the centers of the circumscribed circles of $P_{0,i}$ and $P_{2\pi/b,i}$ and is a translation of the leftmost summit of $P_{0,i}$, $p_{left}$ is in $P_{0,i}$.  It remains to show that its adjacent edges intersect the vertical edge of $P_{0,i}$.
Without loss of generality, we suppose that $z_{-1,y}$ is such that the axis of symmetry of $P_{0,i}$ is on $y=0$, i.e. $y=0$ bisects the vertical edge of $P_{0,i}$. The upper summit $u$ of this vertical edge has for $y$-coordinate $u_y=(\sec(\pi/2-2\pi/b)-\tan(\pi/2-2\pi/b))/2$, half the side length of $P_{0,i}$. Similarly, the lower vertex has for $y$-coordinate the opposite of that value. We know the angles of $P_{2\pi/b,i}$, as such the intersection of the top edge adjacent $p_{left}$ has $y$-coordinate equal to $-p_{left,x}\tan(\pi/2-\pi/b)$. This is equal to

$$\left(0.5-\frac{1}{2\cos(\pi/b)}\right)\frac{\cos(\pi/b)}{\sin(\pi/b)}.$$

We want this to be less than $u_y$. After reformulating, we want
$$\frac{\cos(\pi/b)-1}{2\sin(\pi/b)}+\frac{\cos(2\pi/b)-1}{\sin(2\pi/b)}<0.$$ 

This is clearly true for $b\geq 5$ since both cosines are strictly smaller than 1 and the sines are positive. The upper edge adjacent to $p_{left}$ therefore intersects the vertical edge of $P_{0,i}$. By symmetry, we obtain that the lower edge from $p_{left}$ also intersects this vertical edge. We therefore have that $P_{0,i}\cap P_{2\pi/b}$ is a triangular region with $p_{left}$ as one of the summits, and the two other summits are determined above.
\end{proof}

\begin{proposition}\label{prop:rotation}
    The intersection of all rotated inner $b$-gons $P_{\alpha,i}$ for $\alpha \in [0,2\pi/b)$ is exactly the intersection of $P_{0,i}$ and $P_{2\pi/b,i}$.
\end{proposition}

We note that while there is no polygon with exactly a rotation of angle $2\pi/b$, there is a sequence of polygons whose limit is this rotated polygon. Since we consider only the interior of each polygon as the acceptable region for the origin, this has no impact on our result.

\begin{proof}
    We consider $P_{0,i}$ and $P_{2\pi/b,i}$ the inner polygons of $P_0$ and $P_{2\pi/b}$. We first note that $P_{2\pi/b}$ is exactly a copy of $P_0$, translated to the right by 1. We separate in two cases depending on the parity of $b$. Firstly, we consider an even $b$. In this case, there are exactly two horizontal sides of $P_{0}$: $[z_{-1},z_0]$ and $[z_{b/2-1},z_{b/2}]$. We also have $z_{-1,x}=z_{b/2,x}$ and $z_{0,x}=z_{b/2-1,x}$. The normals associated with $z_{-1}$ and $z_{b/2-1}$ are vertical and the intersection of the desired halfplanes forms a vertical band $B_1=\{(x,y):z_{b/2-1,x}<x<z_{-1,x}\}$.  
    $P_{2\pi/b}$ is the symmetric of $P_{0}$ with respect to the axis $x=z_{-1,x}$. It also has two horizontal segments, $[z'_{0},z_{-1}]$ and $[z_{b/2},z'_{b/2-1}]$, where $z'_{j}$ is the symmetric of $z_j$ w.r.t the axis given above. For $P_{2\pi/b}$, its intersection polygon is contained in the band $B_2=\{(x,y):z_{-1,x}<x<z'_{0,x}\}$. We obtain $P_{0,i} \cap P_{2\pi/b,i}=\emptyset$. This implies that the intersection of all the rotated polygons is also empty.
    
We now turn to the odd $b$ case, where the intersection is not empty. We need to show that for any $\alpha \in [0,2\pi/b)$, $(P_{0,i} \cap P_{2\pi/b,i}) \subseteq P_{\alpha,i}$. For all of these polygons, the translation of the global polygon naturally leads to a translation of this center for $P_{2\pi/b}$. The length of each side of an inner polygon is given by $\sec(\frac{\pi}{2} - \frac{2 \pi}{b}) - \tan(\frac{\pi}{2} - \frac{2 \pi}{b})$. As illustrated in Figure~\ref{fig:intersec} and shown in Lemma~\ref{lemma:triangle}, the intersection between the two translated inner polygons $P_{0,i}$ and $P_{2\pi/b,i}$ is non-empty.

\begin{figure}[h!]
    \centering
    \includegraphics[width=0.5\linewidth]{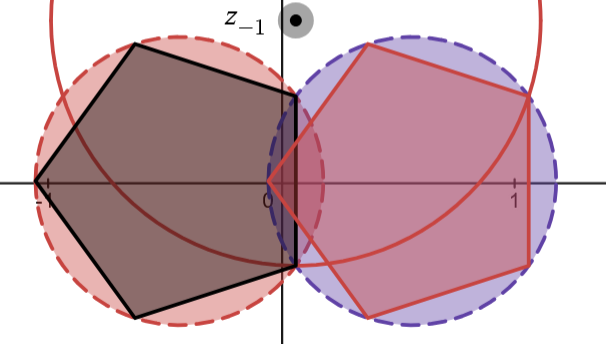}
    \caption{The intersection of the two inner polygons in the case of $b=5$.}
    \label{fig:intersec}
\end{figure}

Let $c$ be a point in $P_{0,i}\cap P_{2\pi/b,i}$. Let $A$ be the arc of angle $-2\pi/b$ centered in $z_{-1}$ and with rightmost extremity in $c$. Since $c \in P_{2\pi/b,i}$, the leftmost extremity of that arc, $d$, is also in $P_{0,i}$.
Observe that $c$ and $d$ can be written
$$c = \left( x_{-1} + r \cos(\theta + \frac{2 \pi}{b}), y_{-1} + r \sin(\theta + \frac{2 \pi}{b}) \right)$$
and
$$d = (x_{-1} + r \cos(\theta), y_{-1} + r \sin(\theta)).$$
Where $r:=|c-z_{-1}|$.
Let $0 < \beta \leq \frac{2 \pi}{b}$ be the angle of rotation of $P_{\beta}$ with respect to $z_{-1}$ (with the extremal case $\beta = \frac{2 \pi}{b}$ corresponding to $P_{\beta} = P_{2\pi/b}$).
Consider 
$$c_\beta = \left( x_{-1} + r \cos(\theta + \frac{2 \pi}{b} + \beta), y_{-1} + r \sin(\theta + \frac{2 \pi}{b} + \beta) \right)$$
and
$$d_\beta = (x_{-1} + r \cos(\theta + \beta), y_{-1} + r \sin(\theta + \beta)).$$
We have that $c$ lies on the arc between $d_\beta$ and $c_\beta$ about $z_{-1}$ 
(with $c = d_{\beta}$ for $\beta = \frac{2 \pi}{b}$).
If we suppose that the arc $A$ lies within $P_{0,i}$, then the $d_\beta$ to $c_\beta$ arc lies within $P_{\beta,i}$.
Since $c$ lies within the second arc, we have that $c \in P_{\beta,i}$.
It remains to show that the arc $A$ is in $P_{0,i}$. We will do this in two steps, first showing that the outermost arc does not intersect the outer edges apart from at its endpoints. 
This furthest arc $A_{\max}$ corresponds to the lowest intersection point between the circumscribed circle of $P_{0,i}$ and $P_{2\pi/b,i}$, and is at distance $K = \sec(\pi/2 - 2\pi/b)$ from $z_{-1}$. We now show that the intersection of $A_{\max}$ and $P_{0,i}$ forms a single arc.

\begin{figure}[h!]
    \centering
    \includegraphics[width=0.4\linewidth]{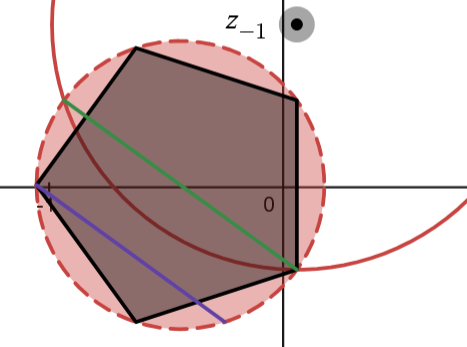}
    \caption{Base 5: in green $s_1$ and in blue $s_2$, the two segments that we use to bound the red arc in $P_{0,i}$.}
    \label{fig:arc}
\end{figure}

As illustrated in Figure~\ref{fig:arc}, let $s_1$ be the segment between the two intersection points of $A_{\max}$ and the circumscribed circle of $P_{0,i}$. Note that it is orthogonal to the left edge of the inner polygon it intersects. Let $s_2$ be the parallel segment passing through the lower summit of the side of $P_{0,i}$ that $A_{\max}$ intersects. The maximum distance between $s_1$ and $A_{\max}$ is reached in the middle of that segment, and is strictly smaller than 
$$\frac{\sec(\pi/2 - 2\pi/b)-\tan(\pi/2 - 2\pi/b)}{2}.$$ The distance between $s_1$ and $s_2$ is equal to $$\frac{\sec(\pi/2 - 2\pi/b)-\tan(\pi/2 - 2\pi/b)}{2}.$$ $A_{\max}$ is therefore always contained between $s_1$ and $s_2$. There are therefore only two sides of $P_{0,i}$ that can intersect $A_{\max}$, exactly those intersected by $s_1$ and $s_2$. We need to verify that each can be intersected only once by the arc. Both of these sides have one point inside the circle of radius $K$ centered in $z_{-1}$ and one outside. As such, there can be only one intersection point with each side. The intersection of $A_{\max}$ and $P_{0,i}$ therefore forms a single arc and since both extremal points of $A_{\max}$, $d$ and $c$, are in $P_{0,i}$, we conclude that $A_{\max} \subseteq P_{0,i}$. For any other arc $A_{\alpha}$, since $c_{\alpha}$ and $d_{\alpha}$ are in $P_{0,i}$ and $s_1$ is orthogonal to the left side of $P_{0,i}$, its orthogonal projection onto the line on which lies $s_1$ is in $s_1$. 
$A_{\alpha}$ is an arc between $d_{\alpha}$ and $c_{\alpha}$ for a circle centered in $z_{-1}$, as such it must lie on the same side of $(c_{\alpha},d_{\alpha})$ as $A_{\max}$.
This implies that $A_{\alpha}$ is in the convex region defined by the arc $A_{\max}$, $c_{\alpha}$, $d_{\alpha}$ and the endpoints of $s_1$. All of these elements are in $P_{0,i}$, which is itself a convex region. We therefore have that $A_{\alpha} \subset P_{0,i}$. This concludes the proof.
\end{proof}

\subsection{Proof of Theorem~\ref{th:polygon}}

\begin{proof}
    We first treat the much easier even $b$ case. As Proposition~\ref{prop:rotation} shows,  the intersection of the acceptable bands is empty, $B_1 \cap B_2 =~\emptyset$. However, we must have $(0,0)\in B_1$ and $(0,0)\in B_2$, we have a contradiction. It is therefore impossible for both these $b$-gons to appear. We do not have exactly $P_{2\pi/b}$ in our set of rotated polygons, but only a sequence of polygons converging to $P_{2\pi/b}$.

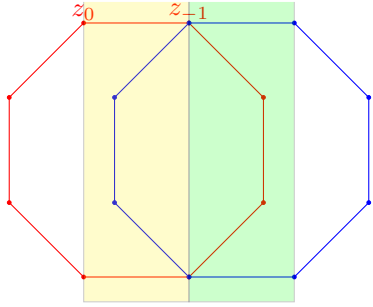
\begin{figure}[h!]
    \centering
    \begin{tikzpicture}[scale=0.7]

    \filldraw[red] (0,2.8) circle (1.pt);
    \filldraw[red] (-2.0, 2.8) circle (1pt);
    \filldraw[red] (-3.414, 1.385) circle (1pt);
    \filldraw[red] (-3.414, -0.614) circle (1pt);
    \filldraw[red] (-2.0, -2.028) circle (1pt);
    \filldraw[red] (0.0, -2.028) circle (1pt);
    \filldraw[red] (1.414, -0.614) circle (1pt);
    \filldraw[red] (1.414, 1.385) circle (1pt);
    \draw[red] (0,2.8) -- (-2.0, 2.8);
    \draw[red] (-2.0, 2.8) -- (-3.414, 1.385);
    \draw[red] (-3.414, 1.385) -- (-3.414, -0.614);
    \draw[red] (-3.414, -0.614) -- (-2.0, -2.028);
    \draw[red] (-2.0, -2.028) -- (0.0, -2.028);
    \draw[red] (0.0, -2.028) -- (1.414, -0.614);
    \draw[red] (1.414, -0.614) -- (1.414, 1.385);
    \draw[red] (1.414, 1.385) -- (0,2.8);
    \node[align=left][red] at (0,3) {$z_{-1}$};
    \node[align=left][red] at (-2.0,3) {$z_0$};
    \filldraw[blue] (0,2.8) circle (1.pt);
    \filldraw[blue] (2.0, 2.8) circle (1pt);
    \filldraw[blue] (3.414, 1.385) circle (1pt);
    \filldraw[blue] (3.414, -0.614) circle (1pt);
    \filldraw[blue] (2.0, -2.028) circle (1pt);
    \filldraw[blue] (0.0, -2.028) circle (1pt);
    \filldraw[blue] (-1.414, -0.614) circle (1pt);
    \filldraw[blue] (-1.414, 1.385) circle (1pt);
    \draw[blue] (0,2.8) -- (2.0, 2.8);
    \draw[blue] (2.0, 2.8) -- (3.414, 1.385);
    \draw[blue] (3.414, 1.385) -- (3.414, -0.614);
    \draw[blue] (3.414, -0.614) -- (2.0, -2.028);
    \draw[blue] (2.0, -2.028) -- (0.0, -2.028);
    \draw[blue] (0.0, -2.028) -- (-1.414, -0.614);
    \draw[blue] (-1.414, -0.614) -- (-1.414, 1.385);
    \draw[blue] (-1.414, 1.385) -- (0,2.8);
    \draw [draw=black, fill=yellow, opacity=0.2]
       (0,3.2) -- (0,-2.5) -- (-2.0,-2.5) -- (-2.0,3.2);
       \draw [draw=black, fill=green, opacity=0.2]
       (0,3.2) -- (0,-2.5) -- (2.0,-2.5) -- (2.0,3.2);
    \end{tikzpicture}
    \caption{In red $P_{a}$ and in blue $P_{b}$. In respectively yellow and green are $B_1$ and $B_2$, which do not intersect as they are open.}
    \label{fig:evenbgon}
\end{figure}
For any $p$ in $P_{0,i}$, there will exist $\delta>0$, such that the polygon $P_{2\pi/b-\delta}$ does not contain $p$ in its inner intersection polygon. As such there is no valid position for $(0,0)$ to obtain a periodic start.We now consider $b$ odd. Using Lemma~\ref{lemma:triangle}, we know that the intersection of $P_{0,i}$ and $P_{2\pi/b,i}$ is a triangle which we characterized in the proof of the Lemma. By Proposition~\ref{prop:rotation}, we know that the intersection of all the rotated polygons is exactly $P_{0,i} \cap P_{2\pi/b,i}$.
Given that the (open) inner polygons correspond to the acceptable regions for the origin, $z_{-1}$ must be such that the origin is contained in the interior of $P_{0,i}\cap P_{2\pi/b,i}$. This corresponds to a first triangular region $T_1$ with $z_{-1}>0$ coordinate-wise. Our construction corresponded to a first step to the left. In other words, all the sign choices during the construction of the sequence were ``-''. As mentioned previously, any sequence starting in $z_{-1}$ with only ``-'' choices is the symmetric with respect to the origin of a sequence starting in $-z_{-1}$ with only ``+'' signs. A second triangular region $T_2=\{(-x,-y):(x,y)\in T_1\}$ is exactly all the valid periodic starting positions for ``+'' signs.\end{proof}

An example of these regions for $b=5$ is provided in Figure~\ref{fig:perio}. The only base not covered by the Theorems shown so far is base 3. The big difference with other higher bases is that the intersection of the halfplanes forms a triangle which contains the initial triangle, rather than contained in. A similar approach could be taken, but Proposition~\ref{prop:rotation} is no longer verified. As such, the intersection of the larger valid triangles would need to be done over the rotation of all triangles. The edge of the valid region would then correspond to the union of three identical curves. One of these corresponds to the convex hull of the rotation of angle $2\pi/3$ around $z_{-1}$ of the line passing through $z_{0}$ such that the angle between $(z_{0},z_{-1})$ and this line (considered with positive slope) is $\pi/6$. A further case distinction would then need to be done depending on the sign of $z_{-1,x}$ and $z_{-1,y}$. We do not provide a formal proof of the region of periodic starts, only an illustration in Figure~\ref{fig:perio}.

    \begin{figure}[h!]
    \centering
    \includegraphics[width=0.45\linewidth]{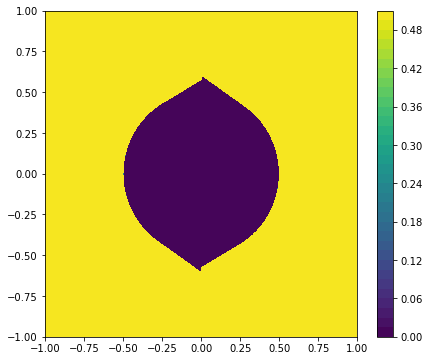}
    \includegraphics[width=0.45\linewidth]{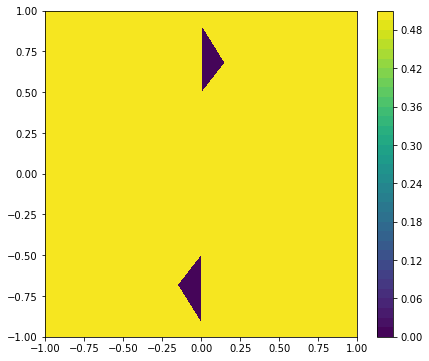}
    \caption{Left: base 3, right: base 5. In purple, regions of $[-1,1]^2$ where the starting position leads to a distance of less than $10^{-4}$ between $z_{-1}$ and $z_{149}$. Other values set to 0.5 for simplicity.}
    \label{fig:perio}
\end{figure}

\subsection{A limit of the characterization}

Finally, even though there are no infinite periodic sequences for even $b$, it is possible to find periodic sequences of length greater than $n$, for any $n \in \mathbb{N}$. Proposition~\ref{prop:endseq} characterizes the length of the sequence $(z_n)_{n \in \mathbb{N}}$ given a starting position close to the best possible choice. 

\begin{proposition}\label{prop:endseq}
Let $b$ be a positive even integer. For $\varepsilon > 0$, if $z_{-1} =  (\varepsilon, \frac{M}{2})$, where $M$ is the height of a regular $b$-gon, then the $b$-gon rotation stops at the $k$-th cycle, where $k$ is the first integer that satisfies $2\pi\Vdc_b(kb)=\arctan(M/(2\varepsilon))-\pi/2+2\pi/b$.
\end{proposition}

Before we prove this proposition, we first verify that this inequality indeed has a solution.
In our case, since $M$ and $\varepsilon$ are positive, we have $\arctan\frac{M}{2\varepsilon} \in (0, \frac{\pi}{2})$. This gives us
\begin{equation*}
    \arctan(M/(2\varepsilon))-\pi/2+2\pi/b< 2\pi/b.
\end{equation*}
Since $\{\Vdc_b(kb):k \in \mathbb{N}\}$ is uniformly distributed in $[0,1/b)$, there exists a $k$ that satisfies the inequality. Figure~\ref{fig:base8fail} provides an illustration of this behavior. Now let us prove the proposition.
\begin{proof}[Proof of Proposition~\ref{prop:endseq}]
    In an even $b$-gon, the halfplane orthogonal to $(z_i,z_{i+1})$ passing through $z_i$ also passes through $z_{i+b/2}$ for any $i$, where numbering is done modulo $b$. As such the halfplane associated with the $b/2+1$-th step passing through $z_{b/2-1}$ also passes through $z_{-1}$. When considering a maximal rotation $2\pi/b$, this halfplane is the one that will not contain $(0,0)$, as shown in the proof of Theorem~\ref{th:polygon}. We therefore need to find the smallest angle defined by $\Vdc_b(kb)$ for some $k \in \mathbb{N}$ such that the line $(z_{b/2-1},z_{-1})$ crosses the $y=0$ axis below 0. The intersection between this line and $y=0$ is a decreasing function for $\theta \in [0,2\pi/b)$.
     For the $k$-th polygon, the angle between $(z_{-1},z_{b/2-1})$ and the horizontal line in $z_{-1}$ is $\theta +\pi/2-2\pi/b$, where $\theta=2\pi\Vdc_b(kb)$.
     From this, $(z_{-1},z_{b/2-1})$ intersects $x=0$ in $(0,0)$ exactly when $\tan(\theta+ \pi/2-2\pi/b)=M/(2\varepsilon)$.
     This corresponds to $2\pi\Vdc_b(kb)=\arctan(M/(2\varepsilon))-\pi/2+2\pi/b$. For any angle greater than this one, the origin will be not be in the acceptable halfplane obtained in the $(b/2-1)$-th summit of the rotated polygon, and the periodicity will end.
\end{proof}
\vspace{-10pt}
\begin{figure}[h!]
    \centering
    \includegraphics[width=0.4\linewidth]{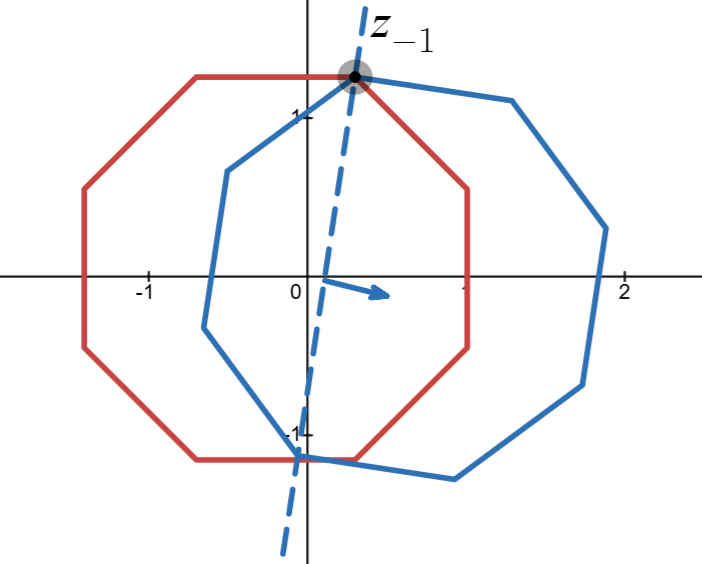}
        \vspace{-5pt}
    \caption{Example where the periodic rotations stop in base 8.}
    \label{fig:base8fail}
\end{figure}

Once again, by central symmetry, an identical result can be obtained for $z_{-1}=(-\varepsilon,-M/2)$.
There is an algorithmic way to find the smallest $kb$ such that the inequality above holds. If we write $kb$ in base $b$, then it is of the form $kb = d_{L} d_{L-1} \cdots d_{1} 0$, where the last digit is $0$. Then $V_{b} (kb)$ is equal to
\begin{equation*}
    V_{b} (kb) = \frac{0}{b} + \frac{d_{1}}{b^2} + \cdots + \frac{d_{L-1}}{b^L} + \frac{d_{L}}{b^{L+1}}
\end{equation*}
Finding the first $d_i<b-1$ then gives us the first polygon which will not be completed. There will be at most, in $b$-ary notation, $(d_i+1)d_{i-1}\ldots,d_10$ $z_{j}$ following the polygonal structure.

\vspace{-5pt}
\end{document}